\documentclass[11pt]{amsart}
\usepackage{amsmath}
\usepackage{graphicx}
\theoremstyle{plain}

\usepackage[usenames]{color}

\newtheorem{thm}{Theorem}[section]
\newcommand{\bt}{\begin{thm}}
\newcommand{\et}{\end{thm}}

\newtheorem{cor}[thm]{Corollary}   

\newcommand{\bc}{\begin{cor}}
\newcommand{\ec}{\end{cor}}

\newtheorem{lem}[thm]{Lemma}   

\newcommand{\bl}{\begin{lem}}
\newcommand{\el}{\end{lem}}

\newtheorem{prop}[thm]{Proposition}
\newcommand{\bp}{\begin{prop}}
\newcommand{\ep}{\end{prop}}

\newtheorem{defn}[thm]{Definition}
\newtheorem{conj}[thm]{Conjecture}

\newcommand{\bd}{\begin{defn}}    
\newcommand{\ed}{\end{defn}}

\newtheorem{rmrk}[thm]{Remark}   

\newcommand{\br}{\begin{rmrk}}
\newcommand{\er}{\end{rmrk}}

\newtheorem{example}[thm]{Example}

\newcommand{\VFto}{\stackrel {\mathcal{VF}}{\longrightarrow} }

\newcommand{\be}{\begin{equation}}
\newcommand{\ee}{\end{equation}}

\newcommand{\R}{\mathbb{R}}

\newcommand{\diam}{\operatorname{diam}}

\newcommand{\area}{\operatorname{Area}}
\newcommand{\vol}{\operatorname{Vol}}

\def\rmin{r_{\mathrm{min}}}

\def\implies{\Longrightarrow}

\def\rmin{r_{\textrm{min}}}

\numberwithin{equation}{section}

\begin{document}

\title[Stability of the Spacetime Positive Mass Theorem]{Stability of the Spacetime Positive Mass Theorem in Spherical Symmetry}



\author[Bryden]{Edward Bryden}
\author[Khuri]{Marcus Khuri}
\address{Department of Mathematics\\
Stony Brook University\\
Stony Brook, NY 11794, USA}
\email{edward.bryden@stonybrook.edu, khuri@math.sunysb.edu}

\author[Sormani]{Christina Sormani}
\address{CUNY Graduate Center and Lehman College\\
New York, NY 10016, USA}
\email{etbryden@gmail.com, sormanic@gmail.com}

\thanks{E. Bryden acknowledges the support of NSF Grant DMS-1612049. M. Khuri acknowledges the support of NSF Grant DMS-1708798. C. Sormani acknowledges the support of NSF Grant DMS-1612049.}

\date{}


\vspace{.2cm}

\begin{abstract}
The rigidity statement of the positive mass theorem asserts that an asymptotically flat initial data set for the Einstein equations with zero ADM mass, and satisfying the dominant energy condition, must arise from an embedding into Minkowski space.
In this paper we address the question of what happens when the mass is merely small. In particular, we formulate a conjecture for the stability
statement associated with the spacetime version of the positive mass theorem, and give examples to show how it is basically sharp if true. This conjecture is then established under the assumption of spherical symmetry in all dimensions. More precisely, it is shown that a sequence of asymptotically flat initial data satisfying the dominant energy condition, without horizons except possibly at an inner boundary, and with ADM masses tending to zero must arise from isometric embeddings into a sequence of static spacetimes converging to Minkowski space in the pointed volume preserving intrinsic flat sense. The difference of second fundamental forms coming from the embeddings and initial data must converge to zero in $L^p$, $1\leq p<2$. In addition some minor tangential results are also given, including the spacetime version of the Penrose inequality with rigidity statement in all dimensions for spherically symmetric initial data, as well as symmetry inheritance properties for outermost apparent horizons.
\end{abstract}

\maketitle

\noindent

\vspace{.5cm}


\section{Introduction}
\label{sec1} \setcounter{equation}{0}
\setcounter{section}{1}

Let $(M^n,g,k)$ be an initial data set for the Einstein equations. This means that
$(M^n,g)$ is a complete Riemannian manifold, possibly with boundary, and $k$ is a symmetric 2-tensor representing the second fundamental form of an embedding into spacetime. These satisfy the constraint equations
\be
	16 \pi \mu=R_g+(\text{Tr}_gk)^2-|k|_g^2, \qquad
	8\pi J=\text{div}_g\left(k-(\text{Tr}_g k)g\right),
\ee
where $\mu$ and $J$ are the energy and momentum density of the matter fields, and $R_g$ denotes scalar curvature. The dominant energy condition is satisfied if
\begin{equation}
	\mu\ge |J|_g.
\end{equation}
We will say that the initial data are \textit{asymptotically flat}
if there is an asymptotic end in the manifold $M^n$ that is diffeomorphic to the complement of a ball $\mathbb{R}^n\setminus B_{0}(\rho_0)$, and there exists a constant $C$ such that in the coordinates $x$ provided by this asymptotic diffeomorphism
\be\label{defn-asym-flat}
\left|\partial^{\beta_1}(g_{ij}-\delta_{ij})\right|\leq\frac{C}{|x|^{n-2+|\beta_1|}},  \qquad\quad
 \left|\partial^{\beta_2} k_{ij}\right|\le \frac{C}{|x|^{n-1+|\beta_2|}},
\ee
for multi-indices $\beta_1\leq2$, $\beta_2\leq 1$ and
\begin{equation}\label{[[[[}
\left|R_{g}\right|\leq\frac{C}{|x|^{n+1}},\qquad\quad
\left|\text{Tr}_g k\right|\leq\frac{C}{|x|^{n}}.
\end{equation}
These fall-off conditions are modeled on those of the original Schoen-Yau proof of the positive mass theorem \cite{SchoenYauII}. We believe our results should follow assuming the weaker asymptotic decay as in the work of Eichmair, Huang, Lee and Schoen \cite{Eichmair,EichmairHuangLeeSchoen}; however, for the sake of simplicity of exposition this will not be done.

With the above setting, the ADM energy and linear momentum of the asymptotic end are finite, well-defined, and given by
\begin{equation}
E=\frac{1}{2(n-1)\omega_{n-1}}\lim_{r\rightarrow\infty}
\int_{S_{r}}(g_{ij,i}-g_{ii,j})\nu^{j},
\end{equation}
\begin{equation}
P_i=\frac{1}{2(n-1)\omega_{n-1}}\lim_{r\rightarrow\infty}
\int_{S_{r}}(k_{ij}-(\text{Tr}_g k)g_{ij})\nu^{j},
\end{equation}
where $S_r$ are coordinate spheres with unit outer normal $\nu$ and $\omega_{n-1}$ is the volume of the standard sphere ${\mathbb S}^{n-1}$.
The ADM mass is then the Lorentz length of the energy-momentum 4-vector
\be
m=\sqrt{E^2-|P|^2}.
\ee
In this paper the main results will be concerned with spherically symmetric initial data. It turns out that in spherical symmetry, under the definition of asymptotic flatness in \eqref{defn-asym-flat} and \eqref{[[[[}, the linear momentum vanishes $|P|=0$ and hence $m=E$ as is shown in Proposition \ref{nbvc}.

The positive mass inequality asserts that an asymptotically flat complete initial data set satisfying the dominant energy condition has
\be
E\geq|P|.
\ee
This was established by Eichmair, Huang, Lee, and Schoen in \cite{EichmairHuangLeeSchoen} for dimensions $3\leq n\leq7$ by using stable marginally outer trapped surfaces (MOTS) in analogy with the minimal hypersurface technique deployed in the time-symmetric case, and in all dimensions $n\geq 3$ for spin manifolds by Bartnik \cite{Bartnik} and Witten \cite{Witten} (see also work of Parker and Taubes \cite{ParkerTaubes}).  Earlier, the weaker inequality $E\geq 0$ was initially proven by Schoen and Yau \cite{SchoenYauII} when $n=3$ with the help of Jang's equation, and this reduction argument was later extended by Eichmair \cite{Eichmair} to include dimensions $3\leq n\leq 7$.

The rigidity of the positive mass theorem may be broken into two statements. The first asserts:
\be
E=|P| \quad \implies \quad E=|P|=0.
\ee
This was proven by Huang and Lee \cite{HuangLee} for $3\leq n\leq 7$. Their approach only uses the positive mass inequality as input but not its proof, and thus can be extended to higher dimensions for spin manifolds. The spin case was previously treated by Beig and Chrusciel \cite{BeigChrusciel} for $n=3$ and Chrusciel and Maerten \cite{ChruscielMaerten} for higher dimensions. The second statement is that
\be
E=0 \implies (M^n,g,k) \textrm{ embeds as initial data in Minkowski space}.
\ee
As with the inequality, this was originally established by Schoen and Yau in \cite{SchoenYauII} for three dimensions and extended by Eichmair in \cite{Eichmair} to dimensions less than eight. Finally in the spin case this was treated for all dimensions in work of Beig, Chrusciel, and Maerten \cite{BeigChrusciel,ChruscielMaerten}. Here we state the positive mass rigidity theorem in a particular way that allows us to propose a natural almost rigidity (or stability) conjecture.

\begin{thm}[Positive Mass Rigidity Theorem \cite{Eichmair,HuangLee,SchoenYauII}]  \label{rigidity-thm} 
Let $(M^n,g,k)$ be a complete asymptotically flat initial data set, with $3\leq n\leq 7$, and satisfying the dominant energy condition. If the ADM mass vanishes $m=0$, then $M^n$ is diffeomorphic to $\mathbb{R}^n$ and
$(M^n,g)$ can be isometrically embedded as a graph in Minkowski space. That is
\be
F:(M^n, g) \to (\mathbb{R}^{1,n}, -dt^2+g_{\mathbb E}), \qquad  F(x)=(f(x),x),
\ee
where $g_{\mathbb{E}}$ is the Euclidean metric and
\be
g= F^*( -dt^2 + {g_{\mathbb E}})=-df^2 + g_{\mathbb E},
\ee
and the second fundamental form, $h$, of the embedding agrees with that of the initial data
\be
h=k.
\ee
\end{thm}

The purpose of this paper is to establish an almost rigidity or stability version of this theorem in the spherically symmetric setting. We will say that the initial data are \textit{spherically symmetric} if $M^n$ is diffeomorphic to $\mathbb{R}^n\setminus B_{0}(r_0)$ or $\mathbb{R}^n$ and the metric and second fundamental form may be
expressed by
\be\label{g-sph-sym}
g=g_{11}(r)dr^2+\rho(r)^2 g_{S^{n-1}},\quad\quad\quad\quad k_{ij}=n_i n_j k_n(r)+(g_{ij}-n_i n_j)k_t(r),
\ee
for some radial functions $g_{11}$, $\rho$, $k_n$, and $k_t$, where $n=\sqrt{g^{11}}\partial_r$ is the unit normal to coordinate spheres. This decomposition for $k$ exhibits its normal and tangential components with respect to the coordinate spheres, and is motivated by the implicit assumption that the initial data come from a spherically symmetric spacetime in which $k$ is the `time derivative' of $g$ which already has this structure.

The boundary, if nonempty, of the initial data will consist of apparent horizons.
Recall that the strength of the gravitational field
around a hypersurface $\Sigma\subset M^n$ may be measured by the null expansions (null mean curvatures) given by
\begin{equation}
\theta_{\pm}:=H_{\Sigma}\pm Tr_{\Sigma}k,
\end{equation}
where $H_{\Sigma}$ is the mean curvature with respect to the unit
normal pointing towards spatial infinity. These quantities can be
interpreted as the rate at which the area of a shell of light changes
as it moves away from the surface in the outward future/past direction ($+$/$-$).
Future or past trapped surfaces are defined by the inequalities $\theta_{+}< 0$ or $\theta_{-}< 0$, respectively, and may be thought of as lying in a region of strong gravity. If $\theta_{+}=0$ or $\theta_{-}=0$, then $\Sigma$ is called a future or past apparent horizon; these naturally arise
as boundaries of future or past trapped regions. Furthermore, such surfaces will be referred to as an \textit{outermost apparent horizon} if it is not enclosed by any other apparent horizon. In Lemma~\ref{sphericalsymmetricAH} it is shown that the outermost apparent horizon inherits the symmetry of its ambient space. In this text the abbreviated term \textit{horizon} will often be used for these objects.

We will consider asymptotically flat $(M^n,g,k)$ that have either no horizons or only a horizon on an inner boundary, in which case the boundary is an outermost apparent horizon. Under these conditions for spherically symmetric initial data, it is shown in
Lemma~\ref{Monotonicity} that the areas (or $n-1$ dimensional volumes) of the level sets of $\rho$ are increasing. Thus we may define the level set
\be \label{Sigma-A}
 \Sigma_A= \rho^{-1}(\rho_A) \textrm{ such that }
 \vol_{g}(\Sigma_A)=A=\omega_{n-1}\rho_A^{n-1}.
\ee
We will study regions within and between these level sets
\be \label{Omega-A}
\Omega_A=\rho^{-1}[0, \rho_A], \quad\qquad \Omega_{A_1,A_2}=\rho^{-1}[\rho_{A_1}, \rho_{A_2}],
\ee
as well as their tubular neighborhoods
\be
T_D(\Sigma_A)= \{p\in M^n \mid\, \exists q\in \Sigma_A \,\,with\,\, d_g(p,q)<D\},
\ee
where $d_g(p,q)$ denotes the distance between $p$ and $q$.

In order to state the main theorem we need the notion of {\em uniform asymptotic flatness}. A sequence of initial data $(M^n_j,g_j,k_j)$ will be referred to as \textit{uniformly asymptotically flat} if each member of the sequence is asymptotically flat according to \eqref{defn-asym-flat} and \eqref{[[[[},
and the constants $\rho_0$ and $C$ in the definition are independent of $j$.

\begin{thm}\label{main-thm}
Fix $A>0$ and $D>\rho_A$, and consider
a sequence of uniformly asymptoticaly flat spherically symmetric initial data sets $(M_j^n, g_j, k_j)$ satisfying the dominant energy condition and with
no closed horizons except possibly the inner boundary. If their ADM masses converge to zero $m_j\rightarrow 0$, then there exist Riemannian manifolds $(\bar{M}^n_j,\bar{g}_j)$ diffeomorphic to $(M^n_j,g_j)$ with graphical isometric embeddings
\be\label{main-thm-graphs}
F_j:(M^n_j, g_j) \to ({\mathbb{R}}\times \bar{M}^n_j, -dt^2+\bar{g}_j), \qquad  F_j(x)=(f_j(x),x),
\ee
\be
g_j = F_j^*( -dt^2 + \bar{g}_j) =-df_j^2 +\bar{g}_j,
\ee
such that the static spacetimes
\be
({\mathbb R}\times \bar{M}^n_j, -dt^2+\bar{g}_j)) \textrm{ converge to Minkowski space }(\mathbb{R}^{1,n} , -dt^2+ g_{\mathbb E})
\ee
in that the base manifolds converge in the pointed volume preserving intrinsic flat sense to Euclidean space. More precisely, regions within $\Sigma_A$ in $(\bar{M}^n_j, \bar{g}_j)$ converge to balls in Euclidean space
\be \label{main-thm-VF}
\left(\,\Omega^j_A \cap T_D(\Sigma^j_A)\, ,\,\bar{g}_j\,\right)\,\, \VFto\,\,
 \left(\,B_0(\rho_A)\,,^{\textcolor{white}{1}} {g_{\mathbb E}}\,\right).
\ee
Furthermore if there is a uniform constant $C$ such that $\parallel k_j\parallel_{L^2(M^n_j)}\leq C$, then for any $1\leq p<2$ the second fundamental forms $h_j$ of the graphs satisfy
\be \label{main-thm-kL2}
\parallel h_j - k_j\parallel_{ L^p \left( \Omega^j_A\cap T_D(\Sigma_A^j), \bar{g}_j\right)} \to 0.
\ee
\end{thm}

The intrinsic flat distance $d_{\mathcal{F}}(\Omega_j, \Omega'_j)$ between pairs of
compact oriented Riemannian manifolds with boundary was first introduced by the third author with Wenger in \cite{SW-JDG}.  Intuitively it measures the filling volume between the given manifolds. It is $0$ if and only if there is an orientation preserving isometry between the manifolds $\Omega_j$ and $\Omega'_j$ \cite{SW-JDG}.
The volume preserving intrinsic flat distance was introduced in \cite{Sormani-Scalar} and includes an extra term involving the global difference of volumes
\be
d_{\mathcal{VF}}(\Omega_j, \Omega'_j)
= d_{\mathcal{F}}(\Omega_j, \Omega'_j) +  |\vol_j(\Omega_j)-\vol_\infty(\Omega'_j)|.
\ee
This has been studied by Portegies in \cite{Portegies-evalues} and by
Jauregui-Lee in \cite{Jauregui-Lee}.  In particular they have shown that
\be
d_{\mathcal{VF}}(\Omega_j, \Omega_\infty)\to 0 \quad\quad\implies\quad\quad \vol(B_{p_j}(r)) \to \vol(B_{p_\infty}(r)),
\ee
for sequences of points $p_j\in \Omega_j$ converging to $p_\infty\in \Omega_\infty$.

Theorem~\ref{main-thm} has been proven for time-symmetric initial data sets by Lee and the third author
in \cite{LeeSormani1}.  In that setting $k_j=0$ and $f_j$ can be taken to be  constant so that $h_j=0$ and
(\ref{main-thm-kL2}) follows trivially. The intrinsic flat convergence is proven in \cite{LeeSormani1} by constructing
an explicit filling manifold.  In fact, LeFloch and the third author have proven in \cite{LeFloch-Sormani} that the metric tensors converge in the
$H^1_{loc}$ sense.
Note that examples in \cite{LeeSormani1} demonstrate that even in the spherically symmetric time-symmetric setting one can have
sequences with ADM mass converging to $0$ which do not converge to regions in Euclidean space in the
smooth or Gromov-Hausdorff sense.
Applying techniques from \cite{LeeSormani1} in our Example~\ref{ex-deep-well}, it is shown why one needs a tubular neighborhood in \eqref{main-thm-VF}. Furthermore, in Example~\ref{ex-two-sheets} and Example~\ref{ex-bubble} we demonstrate the need to assume that there are no interior horizons.

Without time symmetry, when $k_j \neq 0$, Theorem \ref{main-thm} makes no claim as to the convergence of the original sequence of
Riemannian manifolds $(M_j^n, g_j)$.  Example~\ref{ex-to-null} illustrates why
the initial sequence $(M_j^n, g_j)$ need not converge in any reasonable sense.
There we construct sequences of initial data sets of zero mass lying in Minkowski space which become increasingly null on large regions, so that volumes disappear instead of converging.

\begin{conj}\label{rmrk-conj}
Theorem~\ref{main-thm} holds without requiring spherical symmetry when suitable definitions are made for the regions $\Omega_A$.
To achieve the conclusion exactly as stated we expect that $E \to 0$ should replace $m \to 0$ in the hypotheses for the general case.   It is possible that a similar statement holds for $m \to 0$, but the approach would have to be different from the one used here in light of examples with boost.  In the outline below, we clarify which steps strongly use spherical symmetry and which hold more generally.
\end{conj}

The corresponding almost rigidity or stability conjecture in the time-symmetric case was stated and proven in the spherically symmetric setting by Lee and the third author in \cite{LeeSormani1}. It has been confirmed in the graph setting by Huang, Lee, and the third author in \cite{HuangLeeSormani} and for geometrostatic manifolds by the third author with Stavrov in \cite{SormaniStavrov}. Initial controls on the metric tensor towards proving the time-symmetric conjecture have been found by Allen for regions covered by smooth inverse mean curvature flow in \cite{Allen}, and by the first author for axisymmetric manifolds \cite{Bryden}.

With the definition of asymptotic flatness used here the ADM mass $m$ and ADM energy $E$ agree in spherical symmetry since the linear momentum $P$ vanishes (Proposition \ref{nbvc}). In general when mass and energy differ, Conjecture \ref{rmrk-conj} could be quite subtle in the case of large linear momentum, as the construction of the base manifolds $(\bar{M}^n,\bar{g})$ presented here does not behave well in this setting. The methods used here and based on the Jang equation are tailored to the situation when $E$ is small, which will not be the case if $|P|$ stays uniformly away from zero.

We now give an outline of the proof of Theorem \ref{main-thm}, which is modeled on the Schoen-Yau approach to the positive mass theorem \cite{SchoenYauII}. It should be noted that some of the arguments do not require spherical symmetry.
The first step is to solve the 2nd order quasi-linear elliptic Jang equation for each $(M_j^n,g_j,k_j)$ to obtain solutions $f_j$ with asymptotically cylindrical blow-up at the outermost apparent horizon.  See Section~\ref{sec4} for details.
The original study of this equation in \cite{SchoenYauII} observed that the solution only blows-up at apparent horizons. Prescribed blow-up at the outermost horizon was obtained in work of Eichmair, Han, the second author, and Metzger
for low dimensions \cite{Eichmair,HanKhuri,Metzger}.
In the spherically symmetric case, the equation can be reduced to a 1st order ODE and the desired solutions can be produced in any dimension [Theorem~\ref{jangexistence}]. From the solutions a sequence of Riemannian manifolds,
\be
\textrm{ the Jang deformations: } \left(\bar{M}^n_j,\bar{g}_j=g_j+df_j^2\right),
\ee
can be constructed which serve as the base for the ambient static spacetimes of Theorem \ref{main-thm}.
Schoen-Yau showed that the scalar curvature of the Jang metric is nonnegative modulo a divergence term
as stated in (\ref{Jangscalar}).
The Jang manifolds are also uniformly asymptotically flat [Lemma~\ref{Unf-Asym-Flat for bar}], and
have the same ADM masses as the original initial data
[Corollary~\ref{bar-RotSym}]. A primary difference is that they have a cylindrical end where previously there was a boundary.  In Example~\ref{ex-Sch} we explicitly solve the Jang equation for a constant time slice of the Schwarzschild spacetime so that one can see precisely how this step may be implemented constructively.

The nonnegativity property of the scalar curvature of $(\bar{M}_j^n, \bar{g}_j)$ allows one to further conformally transform the Jang deformations to Riemannian manifolds of zero scalar curvature,
\be
\textrm{ the conformal transformations: } \left(\tilde{M}^n_j,\tilde{g}_j=u_j^{4/(n-2)}\bar{g}_j\right).
\ee
Along the cylindrical ends $u_j$ decays exponentially fast to zero and hence conformally closes this end (see \cite{SchoenYauII} and Proposition~\ref{contran}).
In general the masses of the conformal deformations converge to zero. Thus if no horizons are present or one restricts attention to domains outside the outermost minimal surface, it is expected (by the time-symmetric almost rigidity conjecture) that regions in $(\tilde{M}^n_j,\tilde{g}_j)$
converge to balls in Euclidean space in the intrinsic flat sense.
The hope is then to prove that the conformal factors $u_j$ are sufficiently close to $1$ in order to establish that regions in $(\bar{M}^n_j,\bar{g}_j)$ converge to balls in Euclidean space as well. See Remark \ref{rmrk-tilde-general}.

In the spherically symmetric setting the conformal deformation is Euclidean space $\tilde{g}_j=g_{\mathbb{E}}$, so the mass is $0$ and there are no horizons (cf. Lemma~\ref{birkhoff1}).
Therefore $\bar{g}_j$ is related to the Euclidean metric via the conformal factor $u_j$, and establishing \eqref{main-thm-VF} is reduced to controlling $u_j$. A global $L^2$ gradient bound in terms of the mass is obtained from the stability property associated with the Jang surface in Lemma~\ref{Sobolev-control-of-u}, and this is parlayed into $C^{0,\frac{1}{2}}$ control away from the center of the manifold in Proposition~\ref{C0-outside} by using the uniform asymptotically flat assumption and Lemma~\ref{asym-control-on-u-7890}. We then have that $u_j\rightarrow 1$ uniformly on appropriate subdomains avoiding the center.   This work is completed in Section~\ref{sec5}.

The volume preserving intrinsic flat convergence \eqref{main-thm-VF} is proven
in Section~\ref{sec6} in two main steps.  The overall approach is to apply a result of Lakzian and the third author \cite{Lakzian-Sormani}, as stated in Proposition~\ref{prop-main-thm-VFS}, to achieve intrinsic flat convergence.
First, control on $u_j$ is used
to show that regions avoiding the center are smoothly close to annuli in Euclidean space (see Lemma~\ref{prop-subdiffeo-1}).
Secondly, we prove that the volumes of the regions closer to the center are small
using area monotonicity and the coarea formula in Lemma~\ref{lem-vol-Omega-Not-W}.   The remaining required terms are estimated in various lemmas throughout the section.  Moreover, volume convergence follows from the above and is stated in Lemma~\ref{lem-vol-conv}. Without spherical symmetry one might imagine doing something similar, cutting out many wells rather than just the center as in joint
work of the third author with Stavrov in \cite{SormaniStavrov}, or using a completely different approach as in joint work of the third author with Huang and Lee in \cite{HuangLeeSormani}.

In Section~\ref{sec7} convergence of second fundamental forms \eqref{main-thm-kL2} is established, where the proof relies on nonnegativity of the spacetime Hawking mass.
Control on $|h_j-k_j|_{\bar{g}_j}$ away from the center is given in
Proposition~\ref{2ndFundFormOutside} using estimates for the conformal factors $u_j$, as well as the stability property associated with the Jang surface. While $|h_j-k_j|_{\bar{g}_j}$ might be large near the center, with the additional hypothesis on $k_j$ and the small volume inside in Proposition~\ref{L2 bound1}, convergence in the desired tubular neighborhood is achieved in Theorem~\ref{thmhk}.  Finally, in Section~\ref{sec8} we prove Theorem~\ref{main-thm} using all the above.

\smallskip
\noindent{\bf Acknowledgements:} The authors would like to thank Dan Lee for helpful conversations, and Walter Simon for comments. Christina Sormani gratefully acknowledges office space in the Simons Center for Geometry and Physics, Stony Brook University at which most of the research for this paper was performed, and also support from a CUNY Fellowship Leave. Edward Bryden would like to thank the CUNY Graduate Center for allowing
him to visit in Spring 2019 during which this paper was completed.

\section{Examples}
\label{sec2}

In this section we provide some examples which illustrate the importance of
various hypotheses in Theorem~\ref{main-thm}, and some intuition as to what is happening in the proof.  In the time-symmetric setting, where the objects of study are manifolds with nonnegative scalar curvature, examples are given with closed interior horizons
(Examples~\ref{ex-two-sheets} and ~\ref{ex-bubble}) that fail to have
volume preserving intrinsic flat convergence to Euclidean space.

The additional assumption of no closed interior horizons and no boundary is also considered.
In this setting, a time-symmetric initial data set is a graph over itself and the solution to Jang's equation is constant. An example within this context is provided that contains a deep well, demonstrating why tubular neighborhoods
are introduced in order to obtain volume preserving intrinsic flat convergence (Example~\ref{ex-deep-well}).

The proof of Theorem~\ref{main-thm} will then be applied to slices of the Schwarzschild spacetime, so one can see what happens when there is a boundary horizon.  Difficulties arise even in this time-symmetric example because  solutions to Jang's equation blow-up near the horizon so that the base spaces $(\bar{M}^n, \bar{g})$ possess an asymptotically cylindrical end.
In Example~\ref{ex-Sch} we see exactly how Schwarzschild slices arise as graphs over base spaces which are close to Euclidean space in the volume preserving intrinsic flat sense.
This clarifies why the proof of Theorem~\ref{main-thm} is delicate
when the manifolds have boundary.

Finally we consider examples which are not time-symmetric.  In Example~\ref{ex-to-null} it is demonstrated
how even the restriction to sequences of spacelike graphs $(M_j^n, g_j, k_j)$
in Minkowski space does not allow for proper control over the original sequence of Riemannian manifolds $(M_j^n, g_j)$.  This justifies why Theorem \ref{main-thm} only deals with convergence of the base manifolds, and does not address convergence of the given sequence of initial data.

\subsection{Horizons in Time-Symmetric Examples}

The assumption of no closed interior horizons is necessary to avoid
the formation of bubbles and other phenomena which may occur behind a
horizon.  Since the inner boundary is allowed to be a horizon, these hypotheses mean that the main theorem applies within the domain of outer communication. This is consistent with the basic intuition that the ADM mass cannot effectively `see' within a black hole. The following well-known examples explain why one cannot
hope to obtain volume preserving intrinsic flat convergence without these assumptions.

\begin{example}\label{ex-two-sheets}
Riemannian Schwarzschild space is a constant time slice of the Schwarzschild spacetime, with a metric that can be written as
\be \label{eqn-def-Sch}
g=\left(1+ z'(r)^2\right)dr^2 + r^2 g_{S^{n-1}}
\ee
where $z'(r) = \sqrt{2m/(r-2m)}$.  It has a horizon (minimal surface) at $r=2m$, and can be extended smoothly past the horizon by writing $r$ as a function of $z$. In fact the graph is that of a parabola.  If we take a sequence of
Riemannian Schwarzschild spaces of smaller and smaller mass $m$, this parabola becomes more vertical and the vertex decreases to the origin. See Figure ~\ref{fig-two-sheets}.  This sequence of Riemannian Schwarzschild spaces
converges smoothly to Euclidean space on compact sets that avoid the
increasingly thin necks, and by any weak notion of convergence is seen to converge
to a double sheeted Euclidean space.  The volumes of balls centered around
points on the horizon converge to twice the volume of a Euclidean ball.
It is only by removing the part behind the horizon that
one may consider the limit to be a single Euclidean space, and obtain volume
preserving convergence to Euclidean space.
\end{example}

\begin{figure}[h] 
   \centering
   \includegraphics[width=5in]{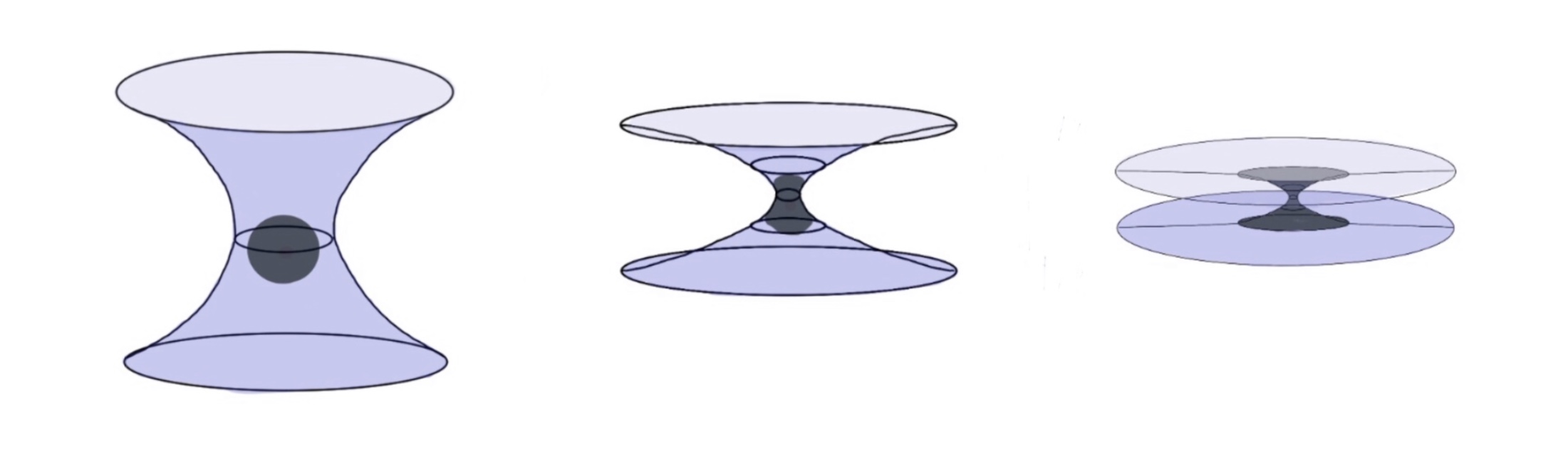}
   \caption{In Example \ref{ex-two-sheets} we see that a sequence of balls centered on the horizons of a sequence of Schwarzschild manifolds with $m_j \to 0$ has volume converging to twice the volume of a Euclidean ball.
      }
   \label{fig-two-sheets}
\end{figure}

\begin{example}\label{ex-bubble}
Start with the Riemannian Schwarzschild initial data viewed as a parabola using
the function $r(z)$ described in the previous example.  Keep the region outside the horizon exactly isometric to Riemannian Schwarzschild, but behind the black hole  attach a round sphere in a $C^1$ way. This is achieved by ensuring that the induced metrics and mean curvatures of the interface surfaces agree from both sides. This guarantees that the scalar curvature is distributionally nonnegative across the interface surface. Furthermore since Riemannian Schwarzschild is scalar flat and the sphere has positive scalar curvature, applying Ricci flow for a very short time, one obtains a smooth metric which is $C^1$ close to the original and has positive scalar curvature everywhere. The almost spherical region of the resulting manifold is called a bubble. See Figure~\ref{fig-bubble}.
Now perform this construction with a sequence of Schwarzschild spaces having masses converging to zero, while keeping the bubble the same size throughout the sequence.  By any notion of weak convergence this sequence converges to a Euclidean space with a sphere attached to it. In analogy with the previous example, balls of a fixed radius centered at points on the horizon have volumes converging
to the sum of the volume of a ball of the same radius in Euclidean space plus a ball of the same radius in the sphere.  Again we do not obtain volume preserving intrinsic flat convergence to Euclidean space, unless the part inside the horizon is cut out.
\end{example}

\begin{figure}[h] 
   \centering
   \includegraphics[width=5in]{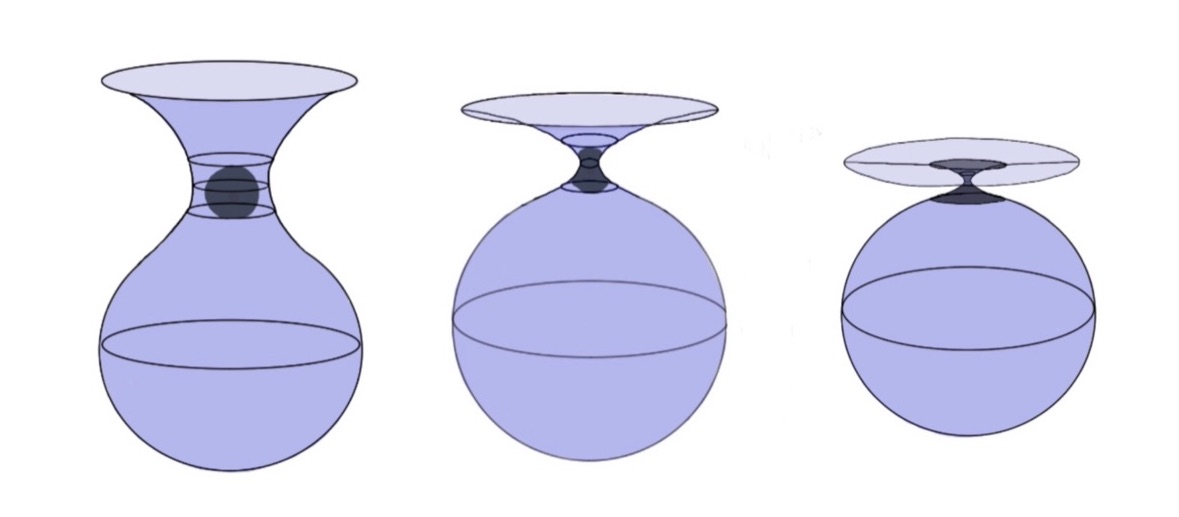}
   \caption{In Example \ref{ex-bubble} we see that a sequence of balls centered on the horizons of a sequence of manifolds with $m_j \to 0$ has volume converging to the sum of the volumes of a Euclidean ball and a ball in a sphere.
      }
   \label{fig-bubble}
\end{figure}

\subsection{Deep Wells in Time-Symmetric Examples}

Let us now consider Theorem~\ref{main-thm} in the time-symmetric case where there are no horizons and no boundary.
In such a setting the solution to Jang's equation is constant, so the theorem states that volume preserving intrinsic flat convergence occurs within the initial data themselves, as opposed to convergence of ambient spacetimes in which the data embed. This was established by D. Lee and the third author in \cite{LeeSormani1}.
In this subsection we first recall in Lemma~\ref{lem-ex} an example construction technique from \cite{LeeSormani1}.
We then review an example, Example~\ref{ex-deep-well}, with deep wells demonstrating the need to use
tubular neighborhoods to obtain volume preserving intrinsic flat convergence.

Recall the definition of Hawking mass for a surface $\Sigma$
in a Riemannian 3-manifold
\be
 m_{H}(\Sigma) = \frac{1}{2}\sqrt{\frac{A}{\omega_{2}}}\left(1-\frac{1}{4\pi}\int_\Sigma \left(\frac{H}{2}\right)^{2}\right),
 \ee
where $A$ and $H$ are the area and mean curvature of $\Sigma$. As described in Section \ref{7890} this may be generalized in spherical symmetry to higher dimensions $n$ by
\be \label{eqn-hawking-1}
m(s) = \frac{\rho^{n-2}(s)}{2}(1-\rho'(s)^2),
\ee
where the metric is expressed in a radial arclength coordinate $s$ and with area radius function $\rho(s)$. The first variation of Hawking mass becomes
\be \label{eqn-hawking-2}
m'(s) = \frac{\rho^{n-1}(s)\rho'(s)}{2(n-1)}R,
\ee
which is nonnegative for nondecreasing area radius functions and nonnegative scalar curvature.

\begin{lem}\label{lem-ex}
Let $\mathcal{M}$ denote the collection of asymptotically flat spherically symmetric manifolds
\be
\left(r^{-1}[r_{min}, \infty)\subset \mathbb{R}^n, \text{ }g=(1+[z'(r)]^2)dr^2 + r^2 g_{S^{n-1}}\right)
\ee
with nonnegative scalar curvature that have no closed interior minimal surfaces and either no boundary, or minimal surface boundary $r^{-1}(r_{min})$.
Let $\mathcal{H}$ be the collection of admissible Hawking mass functions, that is increasing functions
$m:[r_{min},\infty)\to\R$ such that
\be \label{lem-ex-1}
m(\rmin)=\frac{1}{2}\rmin^{n-2},
\ee
and
\be
m(r)<\frac{1}{2}r^{n-2},
\ee
for $r>\rmin\ge 0$.  There is a constructive
bijection between $\mathcal{M}$ and $\mathcal{H}$ such that
\be
m(r)=\frac{r^{n-2}}{2}\left(\frac{z'(r)^2}{1+z'(r)^2}\right)<\frac{1}{2}r^{n-2}.
\ee
\end{lem}


In \cite{LeeSormani1} this result was used to
construct an example with an arbitrarily deep well.  Here we also describe the
volumes in this example, justifying the necessity of cutting off the region using
tubular neighborhoods of fixed size $D$ to obtain volume preserving
intrinsic flat convergence in Theorem~\ref{main-thm}.

\begin{example}\label{ex-deep-well}
Given $A>0$, $L>0$, and $\delta>0$
there exists $(M^n,g)\in \mathcal{M}$
with ADM mass $m< \delta$ such that the
distance $d(\Sigma_{min}, \Sigma_{A}) >L$, where $\Sigma_{A}$ is a symmetry sphere
with area $A$ and $\Sigma_{min}$ is either the boundary $\partial M^n$
or the pole.  See Figure~\ref{fig-deep-well}. In fact the example is constructed by first choosing a radius $r_\epsilon$ depending on $\delta$,
and then constructing the admissible Hawking mass function
so that the distance between the levels $r_\epsilon$ and $r_\epsilon/2$
is an arbitrary value $L$.
Thus the volume between these level sets, computed with
the coarea formula, provides a lower bound
\be
\vol_g(\Omega_A) \ge A_\delta L,
\ee
where $A_\delta$ is the area of the level set $r^{-1}(r_\epsilon)$
which depends on $\delta$ but not $L$.  Taking a sequence
with
\be
\delta_j \to 0 \quad\textrm{ and }\quad L_j=j / A_{\delta_j} \to \infty,
\ee
we obtain a sequence of spherically symmetric examples
which are increasingly deep and have
\be
\vol_{g_j}(\Omega^j_A) \ge \vol_{g_j}(\Omega^j_{A_{\delta_j},A})\to \infty.
\ee
By \eqref{main-thm-VF} of Theorem~\ref{main-thm} it holds that for any fixed $D>0$,
\be
\vol_{g_j}\left(\Omega^j_A \cap T_D(\Sigma^j_A)\right) \to \vol_{g_{\mathbb{E}}}\left(B_0(\rho_A)\right)<\infty.
\ee
This fixed distance $D>0$ of the tubular neighborhood is needed to cut off the arbitrarily large volumes in the arbitrarily deep wells.
\end{example}

 \begin{figure}[h] 
   \centering
   \includegraphics[width=.4\textwidth]{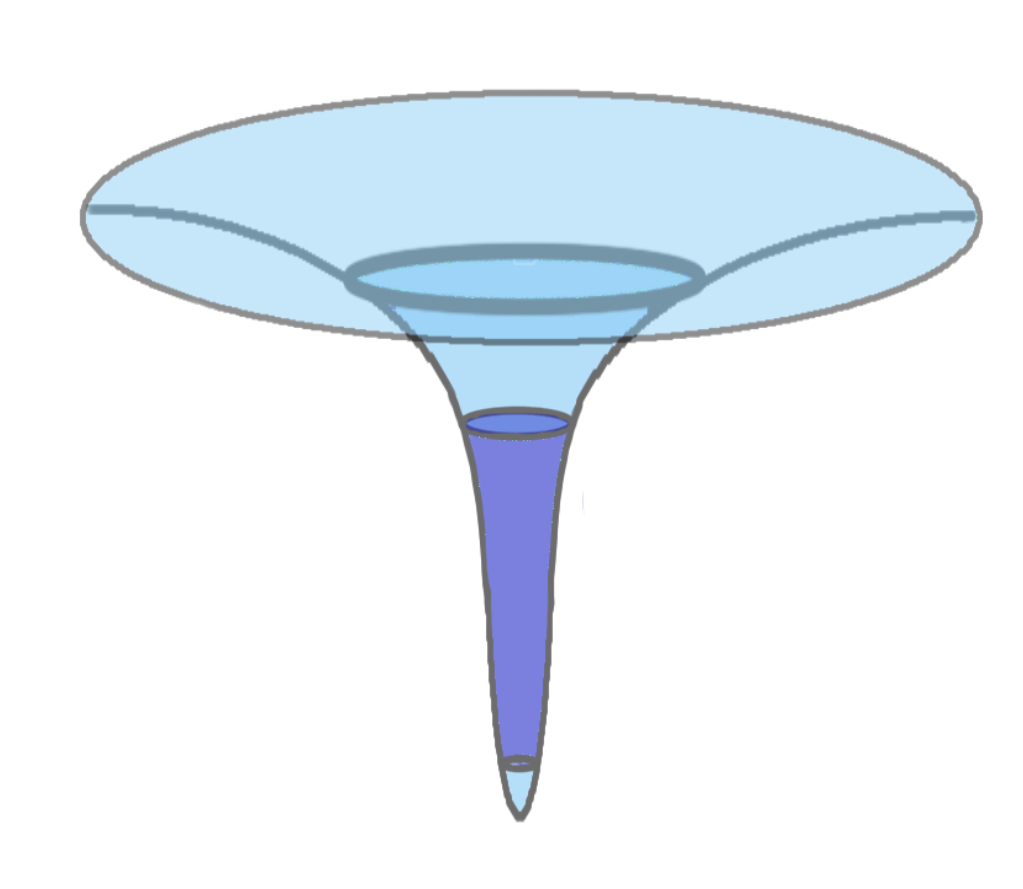}
   \caption{This manifold $(M^n_j,g_j)$ of Example~\ref{ex-deep-well} has a deep well with large volume in the shaded region $\Omega^j_{A_{\delta_j},A}\subset M_j^n$.}
   \label{fig-deep-well}
\end{figure}

\subsection{Riemannian Schwarzschild Space}

Let us now consider Theorem~\ref{main-thm} in the time-symmetric case where there is a boundary horizon; for simplicity we restrict the discussion in this subsection to dimension $n=3$. Even though this setting is time-symmetric the solutions of Jang's equation will not be constant, rather they will blow-up at the horizon boundary. These solutions are then used to embed the initial data $(M^3,g,0)$
into $({\mathbb{R}}\times \bar{M}^3, -dt^2 +\bar{g})$ as a graph over the base space $(\bar{M}^3, \bar{g})$. This base will have different properties than the original data, in particular it will have an asymptotically cylindrical end.  Next an appropriate conformal factor $u$ is found so that the new metric $\tilde{g}=u^4 \bar{g}$ is scalar flat. These are some of the main steps in the proof of Theorem~\ref{main-thm}, and will in this subsection be computed explicitly for Riemannian Schwarzschild initial data.


\begin{example}\label{ex-Sch}
Recall that the induced metric on a time slice $M^3=r^{-1}[2m,\infty)$ of the Schwarzschild spacetime of mass $m$ can be written in the form
\be
g=\left(1-\frac{2m}{r}\right)^{-1}dr^2 +r^2 g_{S^2}.
\ee
The Jang equation may be solved explicitly in this case for a blow-up solution. To see this observe that from \cite{Bray-KhuriDCDS} and the discussion in Section \ref{7788}, Jang's equation may be reduced to a first order ODE by setting
\begin{equation}
v=\frac{\sqrt{g^{11}}f'}
{\sqrt{1+g^{11}f'^{2}}},
\end{equation}
where $g^{11}=1-2m/r$. Namely, the Jang equation in this case becomes simply
\begin{equation}
v'+\frac{2}{r}v=0,
\end{equation}
and the blow-up solution is
$v=\left(\frac{2m}{r}\right)^2$. It follows that
\begin{equation}
g^{11}f'^2=\frac{v^2}{1-v^2}=\frac{1}{\left(r/2m\right)^4-1}
=\frac{1}{(1-2m/r)(1+2m/r)(r/2m)^2[1+(r/2m)^2]},
\end{equation}
so that
\begin{equation}
f'=\frac{(1-2m/r)^{-1}}{(r/2m)\sqrt{(1+2m/r)[1+(r/2m)^2]}}.
\end{equation}
Therefore the Jang metric is
\begin{equation}
\bar{g}=g+df^2=\bar{g}_{11}dr^2+r^2 g_{S^2}=(g_{11}+f'(r)^2)dr^2+r^2 g_{S^2}
\end{equation}
with
\begin{equation}
\bar{g}_{11}=g_{11}+f'(r)^2=(1-2m/r)^{-2}\left[(1-2m/r)+\frac{1}{(r/2m)^2 (1+2m/r)[1+(r/2m)^2]}\right].
\end{equation}
This clearly has an asymptotically cylindrical end as $r\rightarrow 2m$.  Figure~\ref{fig-ex-Sch} illustrates how the
Riemannian Schwarzschild geometry embeds as a graph over this Jang deformation.  Note that it becomes increasingly null upon approach to the horizon.

\begin{figure}[h] 
   \centering
   \includegraphics[width=5in]{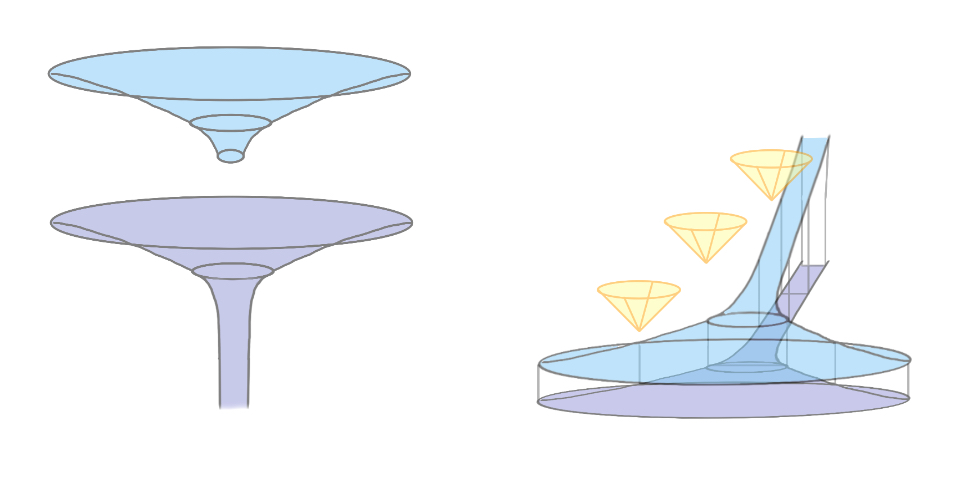}
   \caption{On the left, the Riemannian Schwarzschild geometry (light blue) and its Jang perturbation $(\bar{M}^3,\bar{g})$ in purple are shown as embedded into $\mathbb{E}^4$, with the graphical height $z$ coordinate directed upwards. On the right, the Riemannian Schwarzschild space is illustrated as a graph over the base Jang perturbation. This depiction takes place in 5-dimensional Minkowski space
   ${\mathbb R}^{1,4}\supset({\mathbb R}\times \bar{M}^3, -dt^2 +\bar{g})$,
   where the spatial $z$ coordinate is directed inwards. Light cones are shown in yellow.}
   \label{fig-ex-Sch}
\end{figure}

The conformal deformation to zero scalar curvature can also be given explicitly for this Schwarzschild example. To do this let $r=r(\tilde{r})$ be such that
\be
\bar{g}_{11}\left(\frac{dr}{d\tilde{r}}\right)^2=\left(\frac{r}{\tilde{r}}\right)^2,
\ee
then
\begin{equation}
\bar{g}=\bar{g}_{11}\left(\frac{dr}{d\tilde{r}}\right)^2 d\tilde{r}^2
+\left(\frac{r}{\tilde{r}}\right)^2 \tilde{r}^2 g_{S^2}
=\left(\frac{r}{\tilde{r}}\right)^2\left[d\tilde{r}^2 +\tilde{r}^2 g_{S^2}\right]=\left(\frac{r}{\tilde{r}}\right)^2 {g_{\mathbb E}}.
\end{equation}
We may solve for $\tilde{r}$ in terms of $r$ by
\begin{equation}
\log \tilde{r}=\int_{4m}^r \sqrt{\bar{g}_{11}}r^{-1}dr.
\end{equation}
Now set $u^{-4}=(r/\tilde{r})^2$ so that ${g_{\mathbb E}}=u^4 \bar{g}$. Thus,  $u=\sqrt{\tilde{r}/r}$ serves as the desired conformal factor yielding a scalar flat deformation. The fact that this conformal change resulted in a Euclidean metric is not special to the Schwarzschild example, as will be seen in Section \ref{sec5}.
\end{example}



Figure~\ref{fig-Sch-seq} depicts a sequence of Riemannian Schwarzschild manifolds $(M^3_j,g_j)$ with masses $m_j\rightarrow 0$, embedded as graphs over a sequence of base Jang deformations $(\bar{M}^3_j,\bar{g}_j)$ with asymptotically cylindrical ends that converge in the pointed volume preserving intrinsic flat sense to Euclidean space.

\begin{figure}[h] 
   \centering
   \includegraphics[width=.8\textwidth]{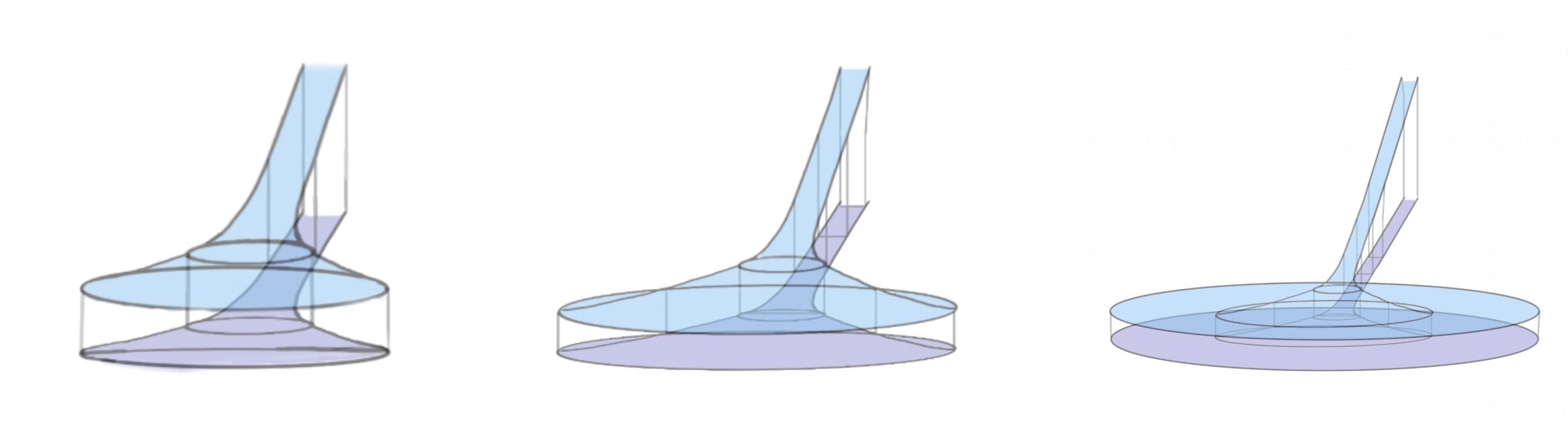}
   \caption{A sequence of Riemannian Schwarzschild manifolds as graphs over their Jang deformations with masses tending to zero. See also Figure \ref{fig-ex-Sch}.}
   \label{fig-Sch-seq}
\end{figure}

\subsection{Graphs in Minkowski Space}

If $(M_j^n, g_j, k_j)$ arise as spacelike graphs $t=f_j(x)$ in Minkowski space,
then the Jang metric $\bar{g}_j$ obtained by solving Jang's equation is exactly the Euclidean metric and $k_j$ is the second fundamental form of the graph. In the notation of Theorem \ref{main-thm} we have
\be
\bar{g}_j=g_{\mathbb E}, \quad\qquad g_j = - df_j^2 + g_{\mathbb E}, \quad\qquad h_j=k_j.
\ee
Theorem~\ref{main-thm} is trivially true and there is nothing to prove. On the other hand, such examples can exhibit pathological behavior from the point of view of establishing volume preserving intrinsic flat convergence of $(M_j^n,g_j)$. In this subsection an example is presented to demonstrate why we say nothing about
the limiting behavior of $(M_j^n, g_j)$ in Theorem~\ref{main-thm}.
In particular, even for sequences of spacelike graphs in Minkowski
space one cannot hope for more than subsequential convergence, and the limiting space need not be well-behaved.

\begin{rmrk} \label{rmrk-Lip}
If $(M_j^n, g_j, k_j)$ arise as spacelike graphs $t=f_j(x)$ in Minkowski space,
then $g_j = - df_j^2 + {g_{\mathbb E}}$ is positive definite, so the Lipschitz norm satisfies $Lip_{g_{\mathbb E}}(f_j) <1$. Thus by Arzela-Ascoli a subsequence of the $f_j$ converge
to a Lipschitz function $f_\infty$ with $Lip_{g_{\mathbb E}}(f_j) \le1$.
However, the graph of $f_\infty$ need not be spacelike!
\end{rmrk}

\begin{figure}[h] 
   \centering
   \includegraphics[width=.4\textwidth]{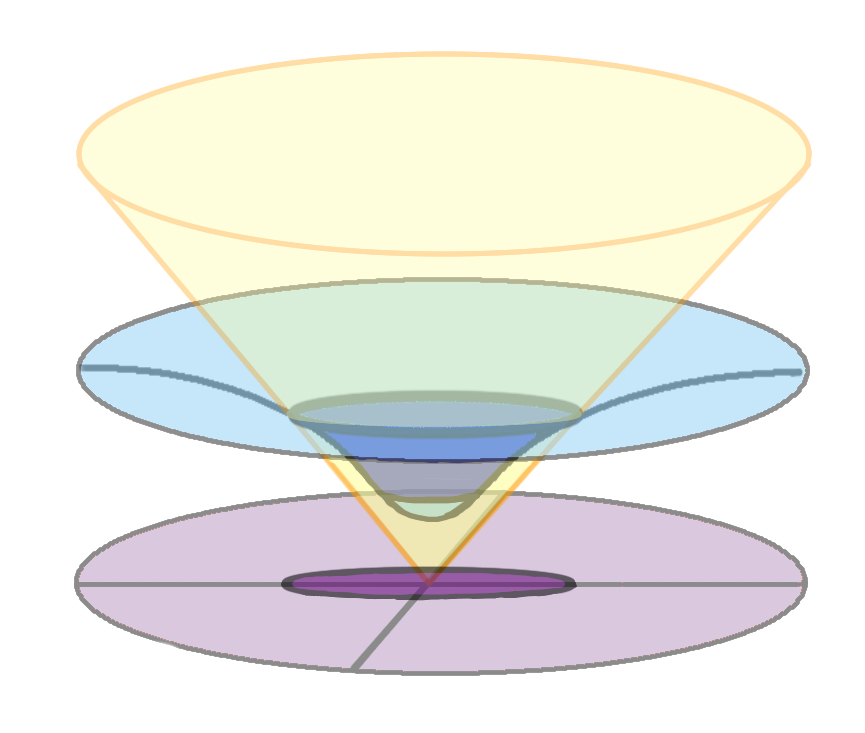}
   \caption{In Example \ref{ex-to-null} we see that even if our given manifold is a graph in Minkowski space the manifold (in blue) can have regions (in dark blue) of very small volume because it is almost null.
      }
   \label{fig-ex-to-null}
\end{figure}

\begin{example}\label{ex-to-null}
Consider a sequence of graphs $t=f_j(x)$ in Minkowski space which are spherically symmetric, with $Lip_{g_{\mathbb E}}(f_j) <1$, and
converging to $f_\infty$ that satisfies
\be
f'_\infty(r)=1 \quad\textrm{ for } \quad r\in [\rho_{A_1}, \rho_{A_2}]
\ee
as in Figure~\ref{fig-ex-to-null}, where $r$ is the radial distance function for $g_{\mathbb{E}}$. It can be arranged that $g_j$ converge in the $C^0$ sense. The limit is a semidefinite metric $g_\infty$ with
\be
g_\infty(\partial_r, \partial_r)= -f'_\infty(r)^2  + 1 =0\quad\textrm{ for } \quad r\in [\rho_{A_1}, \rho_{A_2}].
\ee
By the coarea formula
\be
\vol_{g_j}(\Omega_{A_1, A_2})\le A_2 d_{g_j}(\Sigma_{A_1}, \Sigma_{A_2}) \to 0,
\ee
since
\be
d_{g_j}(\Sigma_{A_1}, \Sigma_{A_2})=\int_{\rho_{A_1}}^{\rho_{A_2}} \sqrt{g_j(\partial_r, \partial_r)}  dr\to 0.
\ee
In contrast
\be
\vol_{\bar{g}_j}(\Omega_{A_1, A_2}) \to \vol_{g_{\mathbb E}}(B_0(\rho_{A_2})\setminus B_0(\rho_{A_1})) \neq 0.
\ee
Thus we find nice behavior of the base spaces $(\bar{M}_j^n, \bar{g}_j)$
as described in Theorem~\ref{main-thm}, with pathological
limiting behavior for the original sequence $(M_j^n, g_j)$.
\end{example}

\section{Manifolds with Spherical Symmetry}
\label{sec3}

In this section we prove that outermost apparent horizons inherit the symmetries of the asymptotically flat initial data sets in which they lie (Lemma~\ref{sphericalsymmetricAH}), that the areas of symmetry spheres are monotonic in spherically symmetric initial data sets without horizons or with outermost apparent horizon boundary (Lemma~\ref{Monotonicity}), and
establish the spacetime Penrose Inequality in all dimensions under the assumption of spherically symmetry (Theorem~\ref{Penrose}).  
These results are of use to us when proving Theorem~\ref{main-thm}.   Prior work in these directions is reviewed within.

\subsection{Horizons in Initial Data With Symmetry}

\begin{lem}\label{sphericalsymmetricAH}
Let $(M^n,g,k)$ be an asymptotically flat initial data set which admits a continuous symmetry with generator $\eta$. This means that $\eta$ is a Killing field which leaves $k$ invariant, and thus the following Lie derivatives vanish $\mathfrak{L}_{\eta}g=\mathfrak{L}_{\eta}k=0$. If the outermost apparent horizon is smooth, then $\eta$ must be tangential to it.

In particular, if $(M^n,g,k)$ is spherically symmetric with smooth outermost apparent horizon then this surface is also spherically symmetric.
\end{lem}

\begin{rmrk}
The existence of outermost apparent horizons due to appropriate trapping is proven in \cite{Andersson-Metzger,Eichmair1}. Like minimal surfaces, they are shown to have a singular set which is no larger than codimension 7. Thus for $2\leq n\leq 7$ the outermost apparent horizon is smooth.
\end{rmrk}

\begin{rmrk}
The symmetry inheritance property for stable MOTS was already observed in Theorem 8.1 of \cite{AnderssonMarsSimon} (see also \cite{PaetzSimon}) when $n=3$. There a spacetime perspective was taken, as opposed to the initial data point of view used here.
\end{rmrk}

\begin{proof}
The following argument is a generalization of that in \cite{BrydenKhuriSokolowsky} for outermost minimal surfaces in axisymmetry. Suppose that the outermost apparent horizon $\Sigma$ does not admit the stated symmetry.
Then the Killing field $\eta$ is not tangential to $\Sigma$ at all points. Thus, if $\varphi_t$ denotes the flow of this Killing field so that $\partial_t\varphi_t=\eta\circ\varphi_t$, then there is a nonzero $t_0$ near zero such that a domain within $\varphi_{t_0}(\Sigma)$ lies outside of $\Sigma$.  Furthermore, observe that since $(M^n,g,k)$ is invariant under the action of $\varphi_t$ the surface $\varphi_{t_0}(\Sigma)$ is an apparent horizon of the same type.

Consider now the compact set $\mathcal{U}$ which is the union of all smooth compact embedded apparent horizons within $M^n$, and define the trapped region $\mathcal{T}$ to be the union of $\mathcal{U}$ with all the bounded components of $M^n\setminus\mathcal{U}$. As described in \cite[Theorem 3.3]{Andersson-Eichmair-Metzger} the outermost apparent horizon arises as the boundary $\partial\mathcal{T}$, moreover it is embedded and smooth away from a singular set of Hausdorff codimension at most 7; in fact it will be smooth by the assumptions of this lemma. Because $\partial\mathcal{T}$ must enclose both $\Sigma$ and $\varphi_{t_0}(\Sigma)$, it cannot agree with $\Sigma$ at all points. This, however, contradicts the outermost assumption for $\Sigma$.
\end{proof}

\subsection{Monotonicity of Area}

In this and the following subsection, the discussion is relevant to the spacetime Penrose inequality in spherical symmetry. The next result may be derived from the arguments in \cite[Section 4]{Mars}.

\begin{lem}\label{Monotonicity}
Let $(M^n,g,k)$ be a spherically symmetric asymptotically flat initial data set as in \eqref{g-sph-sym}. Outside of the outermost apparent horizon, the area of symmetry spheres in $(M^n,g)$ is a strictly increasing function of the radial coordinate. In particular, the warping function $\rho$ defining $g$
is also an increasing function.
\end{lem}

\begin{proof}
Let $S_r$ denote the level sets of $r$. Since
\be \label{k-sph-sym-2}
k=k_n g_{11} dr^2 +k_t \rho^2 g_{S^{n-1}}
\ee
the null expansions (null mean curvatures) are given by
\be
\theta_{\pm}=H\pm \text{Tr}_{S_r}k=(n-1)\left(\sqrt{g^{11}}\frac{\partial_r \rho}{\rho}\pm k_t\right).
\ee
Since the null expansions are both positive near infinity, as the mean curvature $H$ dominates $\text{Tr}_{S_r}k$ according to decay rates, when moving inwards from infinity they must remain positive outside of the outermost apparent horizon where $\theta_+=0$ or $\theta_-=0$. Note that here we are using Lemma \ref{sphericalsymmetricAH} which asserts that the outermost apparent horizon (if present) is one of the spheres $S_{r_0}$. Therefore
\be \label{eq-no-horiz}
\theta_{\pm} >0 \quad\text{ for }\quad r>r_0.
\ee
To finish the proof simply add the two null expansions, and observe that since both are positive we obtain
\be
0<\theta_+ + \theta_-=4\sqrt{g^{11}}\frac{\partial_r\rho}{\rho}\quad\text{ for }\quad r>r_0.
\ee
Hence, the warping function $\rho$ is an increasing function.
\end{proof}


\subsection{The Spacetime Penrose Inequality in All Dimensions}
\label{7890}

The spherically symmetric Penrose inequality without the maximal assumption was established in dimension $n=3$ in \cite{Hayward}, although the case of equality was not treated. A similar result is stated in \cite{Hidayet-et-al} for all dimensions, but the hypotheses are too strong for our purposes and they also do not address the case of equality. The full result including the case of equality was given in \cite{Bray-KhuriDCDS} for $n=3$, and here we easily extend it to all dimensions.

\begin{thm}\label{Penrose}
Let $(M^n,g,k)$, $n\geq 3$ be an asymptotically flat spherically symmetric initial data set satisfying the dominant energy condition, and let $\mathcal{A}_0$ denote the area of the outermost apparent horizon. Then
\be
m\geq \frac{1}{2}\left(\frac{\mathcal{A}_0}{\omega_{n-1}}\right)^{\frac{n-2}{n-1}}
\ee
and equality holds if and only if the initial data outside the outermost apparent horizon arise from an embedding into the Schwarzschild spacetime. In particular, for a sequence of initial data with $m_j\rightarrow 0$ we have $\mathcal{A}_j\rightarrow 0$.
\end{thm}

\begin{proof}
We will follow and generalize the arguments of \cite{Bray-KhuriDCDS} to higher dimensions. The proof is based on the generalized Jang equation introduced in \cite{Bray-KhuriDCDS}. The corresponding Jang deformation is similar to that of the original with the addition of an extra function $\phi$ that plays the role of warping factor for embeddings into a static spacetime, namely $\bar{g}=g+\phi^2 df^2$ where $f$ satisfies equation (3) of \cite{Bray-KhuriDCDS} for a canonical choice of $\phi$.  The scalar curvature of the spherically symmetric generalized Jang metric
\be
\bar{g}=d\bar{s}^{2}+\rho^{2}(\bar{s})g_{S^{n-1}}
\ee
is given by
\be \label{ytr1}
\overline{R}=(n-1)\rho^{-2}[(n-2)(1-\rho_{\bar{s}}^2)-2\rho\rho_{\bar{s}\bar{s}}].
\ee
Here $\bar{s}$ denotes radial distance from the boundary so that $\bar{s}=0$ corresponds to the outermost apparent horizon. The fact that the outermost apparent horizon is a level set of $\bar{s}$ is a consequence of Lemma \ref{sphericalsymmetricAH}.

By comparing arbitrary spherically symmetric metrics to that of Schwarzschild we may derive and generalize the Hawking mass (Misner-Sharp mass in spherical symmetry \cite{MisnerSharp}) to higher dimensions
\be\label{Hawking Mass}
\bar{m}(\bar{s}):=\frac{1}{2}\rho^{n-2}(1-\rho_{\bar{s}}^2)
=\frac{1}{2}\left(\frac{A(\bar{s})}{\omega_{n-1}}\right)^{\frac{n-2}{n-1}}
\left[1-\frac{1}{(n-1)^2 \omega_{n-1}^{\frac{2}{n-1}}A(\bar{s})^{\frac{n-3}{n-1}}}
\int_{S_{\bar{s}}}\bar{H}^2\right],
\ee
where
\be
A(\bar{s})=\omega_{n-1}\rho^{n-1}(\bar{s}),\quad\quad\quad \bar{H}=(n-1)\frac{\rho_{\bar{s}}}{\rho}.
\ee
A direct computation yields
\be
2\bar{m}_{\bar{s}}=\frac{1}{n-1}\rho_{\bar{s}} \rho^{n-1}\bar{R}.
\ee
Therefore integrating produces
\be\label{9090}
\bar{m}(\infty)-\bar{m}(0)=\int_{0}^{\infty}
\frac{\rho_{\bar{s}}\rho^{n-1}}{2(n-1)}\bar{R}d\bar{s}
=\frac{1}{2(n-1)\omega_{n-1}}\int_{\bar{M}^n}\rho_{\bar{s}} \bar{R} dV_{\bar{g}}.
\ee

We may now choose $\phi=\rho_{\bar{s}}$ and follow the arguments in \cite[page 750]{Bray-KhuriDCDS}. This allows one to integrate away the divergence term
appearing in $\bar{R}$, leaving only nonnegative terms on the right-hand side
of \eqref{9090}. It follows that $\bar{m}(\infty)\geq \bar{m}(0)$,
and this gives the desired inequality
\be
m\geq \frac{\rho^{n-2}(0)}{2}
=\frac{1}{2}\left(\frac{\mathcal{A}_0}{\omega_{n-1}}\right)^{\frac{n-2}{n-1}}
\ee
since the ADM mass of the Jang metric agrees with that of the given initial data in addition to the fact that the area of the apparent horizon agrees in both metrics as well. The case of equality follows directly from the arguments of \cite{Bray-KhuriDCDS}.
\end{proof}

\subsection{Decay of the Second Fundamental Form}

We now prove that under the definition of asymptotic flatness for spherical symmetry given in Section \ref{sec1}, the ADM linear momentum vanishes $|P|=0$, and hence the ADM mass coincides with the ADM energy $m=E$.

\begin{prop}\label{nbvc}
Under the asymptotic decay conditions \eqref{defn-asym-flat} and \eqref{[[[[}, a spherically symmetric initial data set $(M^n,g,k)$ satisfies the stronger decay
\begin{equation}
|k|_g=O\left(\frac{1}{|x|^n}\right).
\end{equation}
In particular, the ADM linear momentum vanishes $|P|=0$ and the ADM mass agrees with the ADM energy $m=E$.
\end{prop}

\begin{proof}
Recall that the spherically symmetric initial data $(M^n,g,k)$ may be expressed by
\begin{equation}
g=ds^2+r(s)^2 g_{S^{n-1}},\quad\quad\quad\quad
k_{ab}= n_a n_b k_n +(g_{ab}-n_a n_b)k_t,
\end{equation}
where $n=\partial_s$.
Consider the divergence constraint
\begin{equation}
J=\operatorname{div}_g\left(k-(\text{Tr}_g k)g\right).
\end{equation}
Observe that
\begin{equation}
\nabla_a n_b=\langle\nabla_a n,\partial_b\rangle=2r r'\left(g_{S^{n-1}}\right)_{ab},
\end{equation}
and therefore
\begin{equation}
\nabla_a k_{ab}=2r r'\left[\left(g_{S^{n-1}}\right)_{aa}
(k_n - k_t) n_b +\left(g_{S^{n-1}}\right)_{ab}(k_n -k_t)n_a\right]
+n_a n_b (\partial_a k_n -\partial_a k_t).
\end{equation}
It follows that
\begin{equation}
\left(\operatorname{div}_g k\right)(\partial_b)=
(k_n'-k_t')n_b
+\frac{2(n-1) r'}{r}(k_n-k_t)n_b.
\end{equation}
Furthermore
\begin{equation}
\operatorname{div}_g\left((\text{Tr}_gk)g\right)
(\partial_b)=\partial_b \text{Tr}_gk=\partial_b k_n+(n-1)\partial_b k_t,
\end{equation}
and hence
\begin{equation}
J(\partial_b)=
(k_n'-k_t')n_b
+\frac{2(n-1) r'}{r}(k_n-k_t)n_b
-\partial_b k_n-(n-1)\partial_b k_t.
\end{equation}
The only nonzero component is in the $\partial_s$ direction, and from this we find that
\begin{equation}
k_t'+\frac{2(n-1)r'}{nr}k_t
=\frac{1}{n}J(\partial_s)
+\frac{2(n-1)r'}{nr}k_n.
\end{equation}
Since
\begin{equation}
\text{Tr}_g k=k_n+(n-1)k_t
\end{equation}
this may be rewritten as
\begin{equation}\label{45678}
k_t'+\frac{2(n-1)r'}{r}k_t
=\frac{1}{n}J(\partial_s)
+\frac{2(n-1)r'}{nr}\text{Tr}_g k=:\mathcal{K}.
\end{equation}

A priori the assumed decay for $k$ leads to $k_t=O(s^{1-n})$. However the ODE \eqref{45678} shows that the fall-off is stronger. Indeed, from the dominant energy condition and the assumed decay
\begin{equation}\label{5555}
\text{Tr}_g k=O\left(\frac{1}{s^n}\right),\quad\quad
|J(\partial_s)|_g\leq|J|_g\leq\mu
=\frac{1}{16\pi}\left(R_g+(\text{Tr}_g k)^2 -|k|_g^2\right)=O\left(\frac{1}{s^{n+1}}\right),
\end{equation}
and thus
\begin{equation}\label{0000}
k_t=\frac{1}{r^{2(n-1)}}\left[\int_{s_0}^s r^{2(n-1)}\mathcal{K}ds+C\right]
=O\left(\frac{1}{s^n}\right).
\end{equation}
It then follows from the trace decay in \eqref{5555} that $k_n$ also satisfies
the fall-off in \eqref{0000}, and hence
\begin{equation}
|k|_g=O\left(\frac{1}{s^n}\right),
\end{equation}
which implies that the ADM linear momentum vanishes $|P|=0$.
\end{proof}

\section{Solving Jang's Equation to Obtain the Base Manifolds}
\label{sec4}

\subsection{Review of Jang's Equation Without Symmetry}
Given an initial data set $(M^n,g,k)$, the following quantities may be used to measure how far away it is from being realized as a graph $t=f(x)$ in Minkowski space
\begin{equation}
\bar{g}=g+df^2,\quad\quad\quad\quad
\bar{k}=k-\frac{\nabla_{g}^2 f}{1+|\nabla f|^2_g}.
\end{equation}
In particular, such an embedding exists if and only if $\bar{g}=g_{\mathbb{E}}$ and $\bar{k}=0$. A necessary condition for this to occur is the Jang equation \cite{Andersson-Eichmair-Metzger}
\begin{equation}\label{JangEquation}
\text{Tr}_{\bar{g}}\bar{k}=0\quad\quad\Leftrightarrow\quad\quad
\left(g^{ab}-\frac{f^a f^b}{1+|\nabla f|^2_g}\right)\left(\frac{\nabla_{ab}f}{\sqrt{
	1+|\nabla f|^2_g}}-k_{ab}\right)=0,
\end{equation}
where $f^a=g^{ab}\partial_b f$.
Thus, one may think of Jang's equation as an attempt to find a candidate graph for an embedding into Minkowski space.

Even though such an embedding may not exist for the given initial data, we may use these ideas to construct an isometric embedding into a relevant static spacetime. Namely consider the map
\be
F: (M^n, g) \to ({\mathbb{R}}\times M^n, -dt^2 +  \bar{g})
\ee
defined by
\be
F(x) = (f(x), x).
\ee
Then
\be
F^* ( -dt^2 + \bar{g})= -df^2 + \bar{g}=g,
\ee
and the second fundamental form is
\be
h =  \frac{\nabla^2_{\bar{g}}f}{\sqrt{1-|\nabla f|_{\bar{g}}^2}}=\frac{\nabla_{g}^2 f}{\sqrt{1+|\nabla f|^2_g}}.
\ee

Another motivation for the Jang equation which is pertinent to the positive mass theorem, is to consider it as a method for deforming initial data to obtain weakly nonnegative scalar curvature. In this setting one may view the Jang graph as a submanifold of the $(n+1)$-dimensional dual Riemannian manifold $(\mathbb{R}\times M^n, dt^2+g)$. The Jang metric $\bar{g}$ is then the induced metric on the graph and $h$ is again the second fundamental form. If $k$ is extended trivially off the $t=0$ slice to the whole $(n+1)$-dimensional ambient space, then the Jang equation simply states that the Jang surface $\bar{M}^n$ satisfies the apparent horizon equation
\begin{equation}\label{hgfd}
H_{\bar{M}^n}-\text{Tr}_{\bar{M}^n}k=0,
\end{equation}
where $H_{\bar{M}^n}$ is the mean curvature of the Jang surface.
A computation \cite{SchoenYauII} then shows that the scalar curvature of the Jang metric is nonnegative modulo a divergence term whenever the dominant energy condition is satisfied, that is
\begin{equation}\label{Jangscalar}
R_{\bar{g}}=16\pi(\mu-J(w))+
|h-k|_{\bar{g}}^{2}+2|q|_{\bar{g}}^{2}
-2\mathrm{div}_{\bar{g}}q,
\end{equation}
where
\begin{equation}
w^a=\frac{f^{a}}{\sqrt{1+|\nabla f|_{g}^{2}}},\quad\quad\quad
q_{a}=\frac{f^{b}}{\sqrt{1+|\nabla f|_{g}^{2}}}(h_{ab}-k_{ab}).
\end{equation}
This positivity property for $R_{\bar{g}}$ allows one to conformally transform $\bar{g}$ to zero scalar curvature. Hence through the Jang deformation combined with a conformal transformation, the initial data is taken into the time-symmetric setting.

\subsection{Existence of Solutions to Jang's Equation}
\label{7788}

The Jang deformation preserves uniform asymptotic flatness as will be shown below, and preserves the mass so that $\bar{m}=m$. An interesting feature of the Jang equation's existence theory is its ability to detect apparent horizons. That is, it can only blow-up at apparent horizons in which case it approximates a cylinder over these surfaces \cite{SchoenYauII}. In dimension $n=3$ it has been shown that this cylindrical blow-up behavior can in fact be prescribed at the outermost apparent horizon \cite{Eichmair,HanKhuri,Metzger}. Suppose that the boundary is decomposed into a disjoint union of future ($+$) and past ($-$) apparent horizon components $\partial M^3=\partial_{+}M^3\cup\partial_{-}M^3$, and that there are no other apparent horizons present. If in a neighborhood of $\partial_{\pm} M^3$ there are constants $l\geq 1$ and $c>0$ such that
\begin{equation}
c^{-1}\tau^l \leq\theta_{\pm}(S_{\tau})\leq c\tau^{l},
\end{equation}
where $\tau(x)=\operatorname{dist}(x,\partial M^3)$ and $S_{\tau}$ are surfaces of constant distance to the boundary,
then there exists a smooth solution $f$ of
Jang's equation with the property that
$f(x)\rightarrow\pm\infty$ as $x\rightarrow\partial_{\pm}M^3$. Furthermore, the asymptotics for this blow-up are given by
\begin{align}\label{casymptotics}
\begin{split}
c^{-1}_1\tau^{-\frac{l-1}{2}}+c^{-1}_2 &\leq \pm f
\leq c_1\tau^{-\frac{l-1}{2}}+c_2\quad\quad\textit{ if }\quad\quad
l>1,\\
-c_1^{-1}\log\tau+c_2^{-1} &\leq \pm f
\leq -c_1\log\tau+c_2\quad\quad\textit{ if }\quad\quad l=1,
\end{split}
\end{align}
for some positive constants $c_1$, $c_2$. In the spherically symmetric case, this type of existence result may be established in all dimensions.

\begin{figure}[h] 
\centering
   \includegraphics[width=5in]{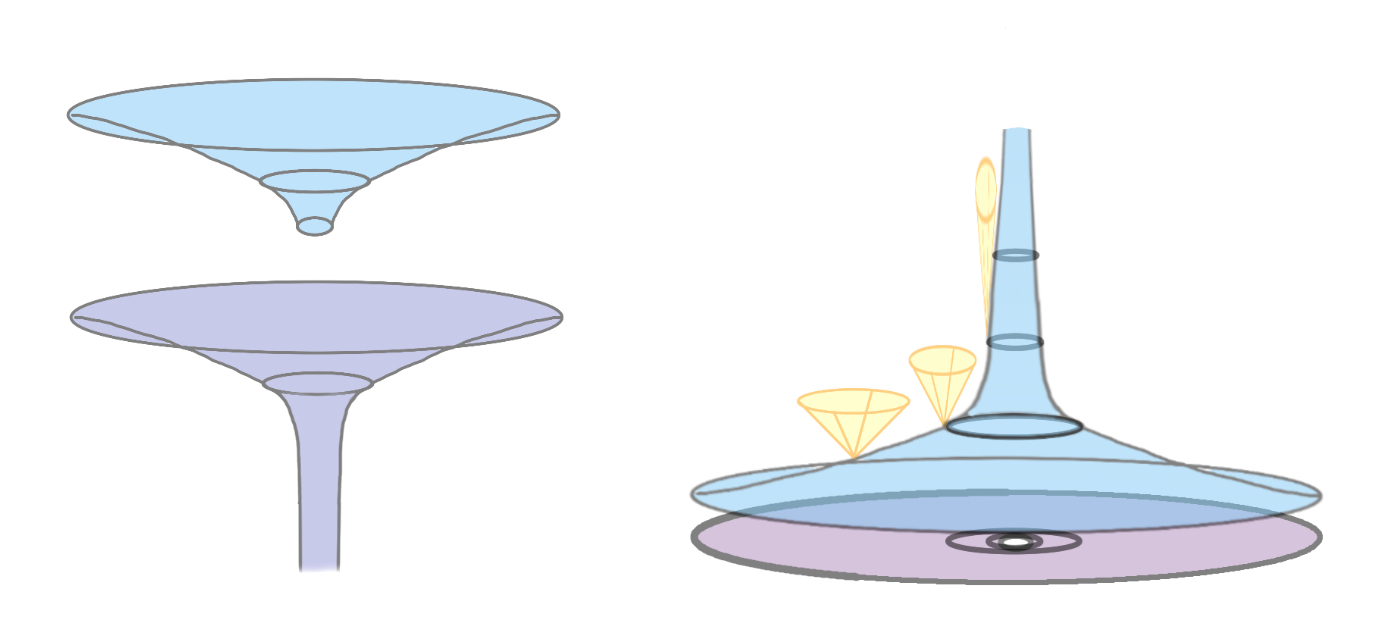}
   \caption{On the left we see the Riemannian Schwarzschild manifold (in blue), and its Jang perturbation (in purple).  On the right we see the Jang perturbation viewed as the base of a static spacetime.  Due to the asymptotically cylindrical end, the light cones of this spacetime are increasingly narrow as we move towards the central cylinder.  Above the base we see the embedding of the Riemannian Schwarzschild space in blue.  The graph of the Jang map has an asymptote at the central cylinder, but it is becoming increasingly null so that the Schwarzschild manifold in blue does not have a cylindrical inner end.}
   \label{jang_blow_up_fig}
\end{figure}

Recall the form of the spherically symmetric initial data
\begin{equation}
g=g_{11}(r)dr^{2}+\rho^{2}(r)g_{S^{n-1}},\text{ }\text{ }\text{
}\text{ }k_{ab}=n_{a}n_{b}k_{n}+(g_{ab}-n_{a}n_{b})k_{t},
\end{equation}
defined on the compliment of a ball
$M^n=\mathbb{R}^{n}\setminus B_0(r_0)$. It is assumed that $\partial M^n=S_{r_0}$ is the only apparent horizon, which means that the null expansions satisfy
\begin{equation}\label{AHC}
\theta_{\pm}(r)=(n-1)\left(\sqrt{g^{11}}\frac{\rho_{r}}{\rho}\pm
k_{t}\right)>0,\text{ }\text{ }\text{ }\text{ }r>r_0,
\end{equation}
and that either
$\theta_{+}(r_0)=0$, $\theta_{-}(r_0)=0$, or
$\theta_{+}(r_0)=\theta_{-}(r_0)=0$ depending on whether $S_{r_0}$ is a
future horizon, past horizon, or both respectively. As observed in \cite{Omurchadha} the Jang equation in spherical symmetry may be reduced to a first order ODE by setting
\begin{equation}
v=\frac{\sqrt{g^{11}}f_{r}}
{\sqrt{1+g^{11}f_{r}^{2}}}.
\end{equation}
The equation \eqref{JangEquation} then becomes
\begin{equation}\label{sphjang}
\sqrt{g^{11}}v_{r}+(n-1)\left(\sqrt{g^{11}}\frac{\rho_{r}}{\rho}v-k_{t}\right)
+(v^{2}-1)k_{n}=0.
\end{equation}
Observe that $|v|\leq 1$ and blow-up occurs precisely when $v=\pm 1$. A maximum principle type argument shows that the outermost horizon condition \eqref{AHC} ensures that $|v|<1$ away from $S_{r_0}$. Building upon this estimate, existence and uniqueness for the spherically symmetric Jang equation may be established following the arguments of \cite[Theorem 2]{Bray-KhuriDCDS}; this prior results was stated for dimension three but the proof carries over to higher dimensions. The result may be stated as follows, under the hypothesis that the initial data satisfy the
following fall-off conditions in the asymptotic end
\begin{align}\label{asymp}
\begin{split}
|k|_{g}=&O_1(r^{1-n}),\quad\quad\quad
\text{Tr}_{g}k=O_1(r^{-n}),\\
g_{11}-1=&O_1(r^{2-n}),\quad\quad\quad
\rho-r=O_2(1).
\end{split}
\end{align}

\begin{thm}\label{jangexistence}
Assume that the initial data set is spherically symmetric, smooth, either complete or with outermost apparent horizon boundary, and satisfies the asymptotics \eqref{asymp}.
Then there exists a unique solution $v\in C^{\infty}((r_0,\infty))\cap
C^{1}([r_0,\infty))$ of \eqref{sphjang} (the spherically symmetric Jang equation) such that $-1<v(r)<1$, $r>r_0$, with
$v(r_0)$ taking the value $0$ or $\pm1$ depending on whether $r_0=0$ and the manifold is complete, or $S_{r_0}$ is a past (future) horizon, respectively.
Furthermore, in the asymptotic end the decay is of the form
\begin{equation}
v=O_2(r^{1-n}),\quad\quad\textit{ as
}\quad\quad r\rightarrow\infty.
\end{equation}
In all cases this gives rise to a spherically symmetric solution $f\in C^{\infty}(M^n)$ of the Jang equation \eqref{JangEquation} satisfying
\begin{equation}
f=O_3(r^{2-n}),\quad\quad\textit{ as
}\quad\quad r\rightarrow\infty.
\end{equation}
\end{thm}

\begin{cor}\label{bar-RotSym}
Under the hypotheses of Theorem \ref{jangexistence} the Jang deformed initial data $(\bar{M}^n,\bar{g})$ is smooth, spherically symmetric, and asymptotically flat with the same mass $\bar{m}=m$. If $(M^n,g,k)$ has a boundary then the Jang deformation, in addition, has an asymptotically cylindrical end satisfying the asymptotics \eqref{casymptotics}.
\end{cor}

\subsection{Uniform Asymptotics for the Solution of Jang's Equation}

The asymptotic fall-off for solutions of Jang's equation given in the previous theorem depend on spherical symmetry and are not necessarily uniform. By allowing for a slightly weaker fall-off we may obtain uniform fall-off in the general case independent of any symmetry.

\begin{lem}\label{Unf-Asym-Flat for bar}
Let $(M^n_j,g_j,k_j)$ be a sequence of uniformly asymptotically flat initial
data, and let $f_j\in C^{\infty}(M^n)$ be corresponding solutions of Jang's equation with $f_j(x)\rightarrow 0$ as $|x|\rightarrow\infty$ where $x\in\mathbb{R}^n\setminus B_{\rho_A}$ are coordinates given by the asymptotic diffeomorphisms. Then for any small $\varsigma>0$ there exist unform constants $\mathcal{C}$ and $\bar{r}$, depending only on $\varsigma$, such that
\begin{equation}\label{hhhh}
|\partial^{\beta}f(x)|\leq \frac{\mathcal{C}}{|x|^{n-2-\varsigma+|\beta|}}\quad\quad\textit{ for
}\quad\quad|x|\geq \bar{r},
\end{equation}
where $\bar{r}>\rho_A$. In particular, the sequence of Jang deformations $(\bar{M}^n_j,\bar{g}_j)$ is uniformly asymptotically flat.
\end{lem}

\begin{proof}
We shall adapt to our purposes an argument of Schoen-Yau which can be found in \cite[pages 248-9]{SchoenYauII} (see also \cite[Proposition 4]{Eichmair}). Let $r(x)=|x|$ and
for $0<p<n-2$, $\lambda>0$, $r> \lambda^{\frac{1}{p+1}}$ define the radial function
\begin{equation}
\bar{f}(r)=\lambda\int^{\infty}_{r} \left(s^{2p+2}-\lambda^2\right)^{-\frac{1}{2}}ds.
\end{equation}
Observe that there is a constant $c_1=c_1(p)$ such that
\begin{equation}
0<\bar{f}(r)\le c_1 \lambda r^{-p},
\end{equation}
and
\begin{equation}\label{uyt}
\frac{d\bar{f}}{dr}\rightarrow-\infty\quad\quad\text{ as }\quad\quad r\rightarrow \lambda^{\frac{1}{p+1}}.
\end{equation}
A computation shows that the Jang operator evaluated at this radial function yields
\begin{equation}\label{hyui}
\left(g^{ab}-\frac{\bar{f}^a \bar{f}^b}{1+|\nabla \bar{f}|^2_g}\right)\left(\frac{\nabla_{ab}\bar{f}}{\sqrt{
	1+|\nabla \bar{f}|^2_g}}-k_{ab}\right)\leq -\lambda (n-2-p) r^{-p-2}
+c_2\left(r^{-n}+\lambda r^{-p-n}\right),
\end{equation}
where $c_2$ is a uniform constant arising from the uniformly asymptotically
flat condition. We may then choose a uniform $\lambda$ large enough to ensure
that the right-hand side of \eqref{hyui} is nonpositive for $r>\lambda^{\frac{1}{p+1}}$,
and thus $\bar{f}$ is a super-solution. Similarly, $-\bar{f}$ is a sub-solution on this domain. Since $f$ and $\bar{f}$ both vanish at spatial infinity, and the derivative \eqref{uyt} is infinity, a maximum principle argument guarantees that $-\bar{f}\leq f\leq\bar{f}$. Therefore
\begin{equation}
|f(x)|\leq c_1 \lambda |x|^{-p}\quad\quad\text{ for }\quad\quad |x|\geq\lambda^{\frac{1}{1+p}}.
\end{equation}
From this, higher order fall-off follows by rescaling combined with the Schauder estimates as in Proposition 3 of \cite{SchoenYauII}. Lastly, we may set $\bar{r}=\lambda^{\frac{1}{p+1}}$ to obtain the statement of this lemma.
\end{proof}

\section{The Conformal Transformations}
\label{sec5}

In this section we construct the conformal transformations and control the conformal factor
as described in the introduction.

\subsection{Review of Conformal Change without Symmetry}

In the previous section we have obtained, from the given initial data $(M^n,g,k)$, a Jang deformation $(\bar{M}^n,\bar{g})$ which is complete, asymptotically flat, and with an additional asymptotically cylindrical end if the original data possessed a boundary. The positivity property \eqref{Jangscalar} for the scalar curvature of the Jang metric leads to a stability-type inequality via integration by parts combined with Cauchy-Schwarz
\begin{equation}\label{stability}
\int_{\bar{M}^n}\left(c_{n}^{-1}|\nabla\phi|_{\bar{g}}^2
+R_{\bar{g}}\phi^2\right)dV_{\bar{g}}
\geq\int_{\bar{M}^n}\left(16\pi(\mu-J(w))+|h-k|_{\bar{g}}^2
+|q|_{\bar{g}}^2\right)\phi^2
dV_{\bar{g}}
\end{equation}
for all $\phi\in C^{\infty}_{c}(\bar{M}^n)$, where $c_{n}=\frac{n-2}{4(n-1)}$. The left-hand side arises from the basic quadratic form associated with the conformal Laplacian $L_{\bar{g}}=\Delta_{\bar{g}}-c_n R_{\bar{g}}$, and asserts that on compact subsets this operator has nonnegative spectrum (for the Dirichlet problem). In fact the spectrum is strictly positive, since if the principal eigenvalue is zero each term on the right-hand side of \eqref{stability} would vanish, implying that $R_{\bar{g}}=0$ and the principal eigenfunction is harmonic, which is impossible.
Thus, a standard exhaustion argument together with asymptotic analysis \cite{Eichmair,SchoenYauII} shows that there is a positive solution $u>0$ of the zero scalar curvature equation
\begin{equation}\label{conformalequation}
L_{\bar{g}}u=0\quad\quad\text{ on }\quad\quad \bar{M}^n,\quad\quad
u(x)=1+\frac{\alpha}{|x|^{n-2}}+O_2\left(\frac{1}{|x|^{n-1}}\right)\quad\quad\text{ as }\quad\quad|x|\rightarrow\infty,
\end{equation}
for some constant $\alpha$. This allows a conformal transformation $\tilde{g}=u^{\frac{4}{n-2}}\bar{g}$ to zero scalar curvature in which the relation between the masses is given by
$
\tilde{m}=\bar{m}+2\alpha
$
\cite[page 259]{SchoenYauII}.

Moreover in the case that $(\bar{M}^n,\bar{g})$ possesses an
additional cylindrical end, the solution $u$ tends to zero in the limit along that end. In fact the decay along the cylindrical end is exponentially fast $u\sim e^{-\gamma s}$, where $s$ is an arclength parameter along the cylindrical end and $\gamma$ is the principal eigenvalue of $\Delta_{\partial M^n}-c_n R_{\partial M^n}$. In spherical symmetry additional assumptions are not required to obtain $\gamma>0$ since the scalar curvature of the outermost apparent horizon $R_{\partial M^n}$ is positive, although in the general case a sufficient condition is for the dominant energy condition to be strict near the horizon as was used in \cite{SchoenYauII}. The next proposition records these observations.

\begin{prop}\label{contran}
Given a smooth Jang deformation $(\bar{M}^n,\bar{g})$ there exists a positive solution $u\in C^{\infty}(\bar{M}^n)$ of \eqref{conformalequation}, so that the conformal metric $\tilde{g}=u^{\frac{4}{n-2}}\bar{g}$ has zero scalar curvature $R_{\tilde{g}}=0$ and $(\tilde{M}^n,\tilde{g})$ is asymptotically flat with mass
\be
\tilde{m}=\bar{m}+2\alpha.
\ee
If an asymptotically cylindrical end is present in the Jang deformation, then
the conformal factor is asymptotic to $e^{-\gamma s}$ along this end with an arclength parameter $s$.
Furthermore, if the initial data are spherically symmetric then the function $u$ and hence metric $\tilde{g}$ are also spherically symmetric.
\end{prop}

\begin{rmrk}\label{rmrk-tilde-general}
Without any symmetry in dimension $n=3$, it follows from a slightly generalized positive mass inequality \cite{SchoenYauII} that $\tilde{m}\geq 0$, and in addition that $\alpha\leq 0$ (see \eqref{stability1} below). Thus if $m=\bar{m}\rightarrow 0$ then $\tilde{m}\rightarrow 0$. The almost rigidity conjecture in the time-symmetric setting then suggests, modulo horizon issues, that $(\tilde{M}^3,\tilde{g})$ is close in the intrinsic flat sense to Euclidean space.
\end{rmrk}

\subsection{Spherically Symmetry Gives $\tilde{g}=g_E$}

We point out that in the spherically symmetric case the existence of a conformal transformation to zero scalar curvature may be obtained from the alternate observation that all spherically symmetric metrics are conformally flat.
This is related to a rigidity phenomena associated to zero scalar curvature in spherical symmetry. The following result, which is similar to Birkhoff's Theorem \cite{Wald} in general relativity, is well-known although the authors do not know of a proper reference in the literature and thus include it here. It should be noted that a proof may be obtained from the arguments of \cite{LeeSormani1}, although here we give another approach.

\begin{lem}\label{birkhoff1}
Let $(M^n,g)$, $n\geq 3$ be spherically symmetric, complete, scalar flat, and asymptotically flat. Then either it is isometric to flat Euclidean space $(\mathbb{R}^n,\delta)$ or the constant time slice of a Schwarzschild spacetime, that is $M^n\cong \mathbb{R}^n\setminus \{0\}$ and
\begin{equation}
g=\left(1+\tfrac{m}{2|x|^{n-2}}\right)^{\frac{4}{n-2}}\delta.
\end{equation}
\end{lem}

\begin{proof}
This conclusion may be derived from the observation that the Hawking mass of radial spheres is constant under the assumption of zero scalar curvature. The inverse mean curvature flow proof of the Penrose inequality \cite{HuiskenIlmanen} then guarantees the desired result where the parameter $m$ is the value of the constant Hawking mass.

An alternative proof is to directly compute the scalar curvature of the spherically symmetric metric in polar form as in \eqref{ytr1}, and analyze the ODE as is done in \cite[page 70]{Petersen}. It is found that $\rho'=1+c \rho^{2-n}$ for some constant $c<0$, the area radius function $\rho(r)>0$ for all $r\in(-\infty,\infty)$, and it has a unique minimum (corresponding to a minimal surface) at $\rho=|c|^{\frac{1}{n-2}}$. It follows that
$\rho$ may be treated as a radial coordinate and so
\begin{equation}
g=dr^2+\rho^2(r)g_{S^{n-1}}=\left(1+\tfrac{c}{\rho^{n-2}}\right)^{-1}d\rho^2 +\rho^2 g_{S^{n-1}}\quad\quad\text{ for }\quad\quad \rho\geq |c|^{\frac{1}{n-2}}.
\end{equation}
Since $\rho(r)$ has a reflection symmetry across the minimal surface, this may be doubled to obtain the stated conclusion.
\end{proof}

This rigidity result suggests that the conformal transformation obtained in Proposition \ref{contran}, in the case of spherical symmetry, gives rise to Euclidean space as we will now see.

\begin{cor}\label{poiuy}
Let $(M^n,g,k)$ be a spherically symmetric, asymptotically flat initial data set satisfying the dominant energy condition which is either complete or has an outermost apparent horizon boundary. Then the conformally transformed Jang deformation of Proposition \ref{contran} is isometric to Euclidean space $(\tilde{M}^n,\tilde{g})\cong(\mathbb{R}^n, g_{\mathbb{E}})$.
\end{cor}

\begin{proof}
We only treat the case with boundary, as the case without boundary is similar. Let $\tau$ be the radial distance function from the boundary for $(M^n,g)$. With the help of \eqref{casymptotics} the Jang metric takes the form
\begin{equation}
\bar{g}=(1+f_\tau^2)d\tau^2+\rho^2 g_{S^{n-1}}=\left(1+c\tau^{-l-1}+O(\tau^{-l})\right)d\tau^2
+\rho^2 g_{S^{n-1}}.
\end{equation}
Set $s=\tfrac{2\sqrt{c}}{l}\tau^{-l/2}$ and use the expansion $\rho=\rho(0)+O(\tau)$ to obtain
\begin{equation}
\bar{g}=\left(1+O(s^{-2-\frac{2}{l}})\right)ds^2
+\left(\rho(0)^2+O(s^{-\frac{2}{l}})\right)g_{S^{n-1}},
\end{equation}
which illustrates the cylindrical asymptotics. Since $u\sim e^{-\gamma s}$, a
straightforward computation shows that the Hawking mass \eqref{Hawking Mass}
of the radial spheres $m(s)\rightarrow 0$ as $s\rightarrow\infty$. Moreover, since $R_{\tilde{g}}=0$, as in the proof of Lemma \ref{birkhoff1} the Hawking mass of these spheres must be constant, and hence zero. The desired conclusion now follows.
\end{proof}

\begin{rmrk}
The results of this subsection rely heavily on the spherically symmetric assumption. Without this hypothesis, the conformally changed manifold is only known to be scalar flat rather than isometric to Euclidean space. Any progress in showing that this manifold is close to Euclidean space will most likely be obtained in conjunction with advances toward the Riemannian version of the stability conjecture.
\end{rmrk}

\subsection{Controlling the Conformal Factor $u$}

For the remainder of this section we will examine how the mass may be
used to control the conformal factor $u$ and the Jang deformation.
Since $\tilde{g}$ is flat its mass vanishes $\tilde{m}=0$, and therefore the formula of Proposition \ref{contran} relating the masses of each deformation yields $-2\alpha=m$; where we have also used that the Jang transformation preserves mass. Furthermore,
multiplying equation \eqref{conformalequation} through by $u$ and integrating by parts, and using the divergence structure present in $R_{\bar{g}}$ yields
\begin{equation}\label{stability1}
-c_{n}^{-1}(n-2)\omega_{n-1}\underbrace{\alpha}_{-\frac{m}{2}}
\geq
\int_{\bar{M}^n}\left[\frac{4}{n-2}|\nabla u|_{\bar{g}}^2
+\left(16\pi(\mu-J(w))+|h-k|_{\bar{g}}^2
+|q|_{\bar{g}}^2\right)u^2\right]
dV_{\bar{g}}.
\end{equation}
It follows that $L^2$ gradient bounds for the conformal factor are given in terms of the mass.

\begin{lem}\label{Sobolev-control-of-u}
Assume that the given spherically symmetric initial data $(M^n,g,k)$ satisfies the dominant energy condition. If $u$ is the solution of \eqref{conformalequation} given by Proposition \ref{contran}, then
\begin{equation}\label{jjjj}
\parallel \nabla u\parallel_{L^2(\bar{M}^n,\bar{g})}^2
\leq\frac{(n-2)^2 \omega_{n-1}}{8 c_{n}}m.
\end{equation}
\end{lem}

\begin{rmrk}
We remark that a version of Lemma \ref{Sobolev-control-of-u} is likely to hold without the assumption of spherical symmetry, when a smooth Jang deformation exists. Namely, if the positive mass theorem is valid for $(\tilde{M}^n,\tilde{g})$ then $-2\alpha\leq m$, and \eqref{jjjj} again follows from \eqref{stability1}. The issue is that the $n$-dimensional Riemannian positive mass theorem \cite{SchoenYauI,SchoenYau}
is not immediately applicable, as $(\tilde{M}^n,\tilde{g})$ may not be a smooth manifold, geometrically or topologically. The later pathology arises from the fact that in higher dimensions, although the outermost apparent horizon is of positive Yamabe invariant, it may not be of spherical topology. Nonetheless, Eichmair \cite{Eichmair} has found an effective way to deal with these concerns.
\end{rmrk}

Next observe that the $L^2$ gradient bound for $u$ on the Jang surface may be translated into a similar bound for $\log u$ on Euclidean space. This follows from Corollary \ref{poiuy} since $g_{\mathbb{E}}=u^{-\frac{4}{n-2}}\bar{g}$, and in particular
\begin{equation}
|\nabla u|^2_{\bar{g}}=u^{\frac{4}{n-2}}|\nabla u|^2_{g_{\mathbb{E}}},\quad\quad\quad\quad
dV_{\bar{g}}=u^{-\frac{2n}{n-2}}dV_{g_{\mathbb{E}}}.
\end{equation}

\begin{cor}\label{Sobolev Bound on omega}
Under the hypotheses of Lemma \ref{Sobolev-control-of-u}
\begin{equation}\label{iiii}
\parallel \nabla \log u\parallel_{L^2(\mathbb{R}^n,g_{\mathbb{E}})}^2
\leq\frac{(n-2)^2 \omega_{n-1}}{8 c_{n}}m.
\end{equation}
\end{cor}

The global Sobolev bounds for the conformal factor obtained from the stability inequality can be parlayed into $C^0$ and even H\"{o}lder estimates away from
the central fixed point of the spherical symmetry. To accomplish we will first need to obtain uniform control for $u$ in the asymptotically flat end.

\begin{lem}\label{asym-control-on-u-7890}
Let $(M^n_j,g_j,k_j)$ be a sequence of spherically symmetric, uniformly asymptotically flat initial data satisfying the dominant energy condition and
with either outermost apparent horizon boundary or no boundary. Let $(\bar{M}_j^n,\bar{g}_j)$ be the corresponding Jang deformations conformally related to $(\tilde{M}^n_j,\tilde{g}_j)$ via conformal factors $u_j$ solving \eqref{conformalequation}. Then there exist uniform constants $c$ and $\bar{r}$ such that
\begin{equation}\label{nbnb}
u_j(r)\leq \exp{cr^{-2(n-2)}}\quad\quad\text{ for }\quad\quad r\geq \bar{r},
\end{equation}
where $r=|x|$ is the radial coordinate from \eqref{defn-asym-flat}
and $\bar{r}$ is as in Proposition \ref{Unf-Asym-Flat for bar}.
\end{lem}

\begin{proof}
For convenience, within the proof the subscript $j$ will be suppressed. Coordinate spheres in the asymptotic region will be denoted by $S_r$; note that they are distinct from the surfaces $\Sigma_A$ used in other parts of the manuscript. Let $\bar{M}^n_r$ denote the component of $\bar{M}^n$ lying inside $S_r$. 
Multiply equation \eqref{conformalequation} through by $u$ and integrate by parts up to a coordinate sphere $S_r$, $r\geq\bar{r}$, and use the divergence structure present in $R_{\bar{g}}$ as in \eqref{stability1} to find
\begin{align}
\begin{split} 
0=& -\int_{\bar{M}^n_r}u\left(c_n^{-1}\Delta_{\bar{g}}u-R_{\bar{g}}u\right)dV_{\bar{g}}\\
=&\int_{\bar{M}^n_r}\left(c_n^{-1}|\nabla u|^2_{\bar{g}}+R_{\bar{g}}u^2\right)dV_{\bar{g}}-\int_{S_r}c_n^{-1}u\partial_{\bar{\nu}}u dA_{\bar{g}}\\
\geq&\int_{\bar{M}^n_r}\left(c_n^{-1}|\nabla u|^2_{\bar{g}}+2u^2 |q|_{\bar{g}}^2
-2u^2\mathrm{div}_{\bar{g}}q\right)dV_{\bar{g}}-\int_{S_r}c_n^{-1}u\partial_{\bar{\nu}}u dA_{\bar{g}}\\
\geq& -\int_{S_r}\left(\tfrac{1}{2c_n}\partial_{\bar{\nu}}u^2+2q(\bar{\nu})u^2\right)
dA_{\bar{g}},
\end{split}
\end{align}
where $\bar{\nu}$ is the unit outer normal with respect to $\bar{g}$. Since all quantities are spherically symmetric and $r$ is arbitrary, we obtain the differential inequality
\begin{equation}
\partial_r \log u\geq -2 c_n q(\partial_r)\quad\quad\text{ for }\quad\quad
r\geq \bar{r}.
\end{equation}
According to the uniform fall-off \eqref{asymp} and \eqref{hhhh}, $q(\partial_r)$ may be estimated to yield
\begin{equation}
\partial_r \log u\geq -\bar{c} r^{-2n+1+2\varsigma}
\end{equation}
where $\bar{c}>0$ is a uniform constant. Integrating on the interval $[r,r_1]$ produces
\begin{equation}
\log u(r_1)-\log u(r)\geq\frac{\bar{c}}{2n-2-2\varsigma}\left(r_1^{-2n+2+2\varsigma}
-r^{-2n+2+2\sigma}\right).
\end{equation}
Now let $r_1\rightarrow\infty$ and use that $u(r_1)\rightarrow 1$ to obtain
\begin{equation}
\log u(r)\leq cr^{-2(n-2)},
\end{equation}
with $c=\bar{c}(2n-2-2\varsigma)^{-1}$.
\end{proof}

Let $\tilde{r}$ denote the radial distance function for $g_{\mathbb{E}}=\tilde{g}$, so that
\begin{equation}
g_{\mathbb{E}}=d\tilde{r}^2 +\tilde{r}^2 g_{S^{n-1}}.
\end{equation}
Consider annular domains $\Omega_{\tilde{A}_1,\tilde{A}_2}\subset\mathbb{E}^n$ whose boundary consists of two coordinate spheres having areas $\tilde{A}_1<\tilde{A}_2$. The next result shows that the conformal factors defining $\tilde{g}$ are uniformly close to $1$ in H\"{o}lder space away from the center of the spherical symmetry. Note that in light of Example \ref{ex-deep-well} we see that it is not possible to have such $C^0$ control near the center.

\begin{prop}\label{C0-outside}
Let $(M^n_j,g_j,k_j)$ be a sequence of spherically symmetric, uniformly asymptotically flat initial data satisfying the dominant energy condition and
with either outermost apparent horizon boundary or no boundary. Let $(\bar{M}_j^n,\bar{g}_j)$ be the corresponding Jang deformations conformally related to $(\tilde{M}^n_j,\tilde{g}_j)$ via conformal factors $u_j$ solving \eqref{conformalequation}. Let $\tilde{A}_{0}>0$ be fixed, assume $m_j<\tilde{A}_0^{-2}$ is uniformly small, and set $\tilde{A}_j= \frac{1}{\sqrt{m_j}}$. Then there exists a uniform constant $C$ such that
\begin{equation}
\parallel\log u_j\parallel_{C^{0,\frac{1}{2}}
\left(\Omega_{\tilde{A}_{0},\tilde{A}_j},g_{\mathbb{E}}\right)}
\leq \frac{C m_j^{\frac{1}{4}}}{\sqrt{\tilde{A}_{0}}}.
\end{equation}
\end{prop}

\begin{proof}
For convenience the subscript $j$ will be suppressed in the proof.
Estimate \eqref{iiii} shows that for any $\tilde{r}_1<\tilde{r}_2$ we have
\begin{equation}
\tilde{r}_1^{n-1}\int_{\tilde{r}_1}^{\tilde{r}_2}\left(\partial_{\tilde{r}}\log u\right)^2 d\tilde{r}\leq \tilde{c}_n m,
\end{equation}
where $\tilde{c}_n=(n-2)^2 c_n^{-1}/8$. With the help of H\"{o}lder's inequality it follows that
\begin{align}\label{rrrr}
\begin{split}
|\log u(\tilde{r}_2)-\log u(\tilde{r}_1)|\,\,\,
=&\,\,\,\left|\int_{\tilde{r}_1}^{\tilde{r}_2}\partial_{\tilde{r}}\log u d\tilde{r}\right|\\
\leq&\,\,\,|\tilde{r}_2-\tilde{r}_1|^{\frac{1}{2}}
\left(\int_{\tilde{r}_1}^{\tilde{r}_2}\left(\partial_{\tilde{r}}\log u\right)^2 d\tilde{r}\right)^{\frac{1}{2}}\\
\leq& \,\,\,\left(\frac{\tilde{c}_n m}{\tilde{r}_1^{n-1}}|\tilde{r}_2-\tilde{r}_1|\right)^{\frac{1}{2}}.
\end{split}
\end{align}
If
\begin{equation}
\tilde{r}_{0}
=\left(\frac{\tilde{A}_{0}}{\omega_{n-1}}\right)^{\frac{1}{n-1}}
\end{equation}
denotes the area radius for the inner boundary of $\Omega_{\tilde{A}_{0},\tilde{A}_j}$ then
\begin{equation}
\frac{|\log u(\tilde{r}_2)-\log u(\tilde{r}_1)|}
{|\tilde{r}_2-\tilde{r}_1|^{\frac{1}{2}}}\,\,\leq\,\, \sqrt{\frac{\tilde{c}_n m}{\tilde{r}_{0}^{n-1}}}
\,\,\,= \,\, \sqrt{\frac{\tilde{c}_n \omega_{n-1} m}{\tilde{A}_{0}}}
\end{equation}
for $\tilde{r}_1\geq \tilde{r}_{0}$,
which yields one half of the desired H\"{o}lder estimate.

The next goal is to obtain $C^{0}$ bounds. Note that \eqref{rrrr} implies
\begin{equation}\label{jik}
|\log u(\tilde{r}_1)|\,\,\leq\,\, |\log u(\tilde{r}_2)|
+\sqrt{\frac{\tilde{c}_n m \tilde{r}_2}{\tilde{r}_{0}^{n-1}}}.
\end{equation}
In order to control $u(\tilde{r}_2)$ uniformly we will utilize Lemma \ref{asym-control-on-u-7890}.
The estimate there, however, is given in terms of the radial coordinate $r$ associated with the uniform asymptotic coordinates of $\bar{g}$. The two coordinates $r$ and $\tilde{r}$ may be compared in the asymptotic end by relating the volumes of coordinate spheres. Let $S_r$ denote a coordinate sphere of radius $r$, then uniform asymptotic flatness shows that its volume with respect to $\bar{g}$ satisfies
\begin{equation}
|S_r|_{\bar{g}}\,\,\leq\,\,\omega_{n-1}r^{n-1}+c_1 r^{n-2}\quad\quad\text{ for }\quad\quad
r\geq \bar{r},
\end{equation}
for some uniform constant $c_1$ (uniform constants will be denoted by $c_i$, $i=1,2,\ldots$). With the help of \eqref{nbnb}, the volume of this same sphere computed with respect to $g_{\mathbb{E}}=u^{\frac{4}{n-2}}\bar{g}$ may be estimated by
\begin{align}
\begin{split}
|S_r|_{g_{\mathbb{E}}}\,\,\,=&\,\,\,u^{\frac{2(n-1)}{n-2}}|S_r|_{\bar{g}}\\
\leq & \,\,\,e^{\frac{2c(n-1)}{n-2}r^{-2(n-2)}}\left(\omega_{n-1}r^{n-1}+c_1 r^{n-2}\right)\\
\leq &\,\,\,\omega_{n-1}r^{n-1}+c_2 r^{n-2}.
\end{split}
\end{align}
On the other hand
\begin{equation}
|S_r|_{g_{\mathbb{E}}}\,=\,\omega_{n-1}\tilde{r}^{n-1},
\end{equation}
and hence
\begin{equation}\label{tttt}
\tilde{r}\,\leq\, r+c_3 \quad\quad\quad \Rightarrow \quad\quad\quad
\frac{1}{r}\,\leq\,\frac{1}{\tilde{r}}+\frac{c_4}{\tilde{r}^2}.
\end{equation}
We may now combine \eqref{nbnb}, \eqref{jik}, and \eqref{tttt} to find for large $\tilde{r}_2$ that
\begin{align}
\begin{split}
|\log u(\tilde{r}_1)|\,\,\,\leq&\,\,\, \frac{c}{r(\tilde{r}_2)^{2(n-2)}}
+\sqrt{\frac{\tilde{c}_n m \tilde{r}_2}{\tilde{r}_{0}^{n-1}}}\\
\leq &\,\,\,\frac{c_5}{\tilde{r}_2^{2(n-2)}}
+\sqrt{\frac{\tilde{c}_n m \tilde{r}_2}{\tilde{r}_{0}^{n-1}}}.
\end{split}
\end{align}
By choosing $\tilde{r}_2=m^{-\frac{1}{2}}$ the desired result follows.
\end{proof}

\section{Intrinsic Flat Convergence of the Base Manifolds}
\label{sec6}

In this section we prove the following proposition which implies
(\ref{main-thm-VF}) of Theorem~\ref{main-thm}.  Recall that
\be
\Omega^j_A=\{ x\in M_j^n\mid\, \rho_j(x)\le \rho_A\}
\ee
is the region within the level set $\Sigma^j_A$ of area $A$ with respect to $g_j$ and
$\bar{g}_j$.  A priori we do not know the area of the level set $\Sigma^j_A$ with respect to the
metric $\tilde{g}_j=u_j^{\frac{4}{n-2}} \bar{g}_j=g_{\mathbb E}$.  This section strongly uses spherical symmetry and
the fact that the metrics $\bar{g}$ and $g$ are monotone as was proven in Lemma~\ref{Monotonicity}.

\begin{prop} \label{prop-main-thm-VFS}
Given any $A>0$, $D>0$, and $\epsilon>0$
there exists $\delta=\delta(A, D,\epsilon)>0$ such that if mass $m_j < \delta$
then
\be\label{prop-main-thm-VF}
d_{\mathcal{VF}}\left(\left(\Omega_j ,\bar{g}_j\right),
\left(\Omega'_j, g_{\mathbb E}\right) \right) < \epsilon,
\ee
where
\be \label{Omega-j}
\Omega_j =\Omega^j_A \cap T_D(\Sigma^j_A) \subset \bar{M}_j^n \textrm{ with metric tensor } \bar{g}_j
\ee
and
\be\label{Omega'-j}
\Omega'_j=  \Omega^j_A \subset {\mathbb{E}}^n \textrm{ with metric tensor } \tilde{g}_j=u^{\frac{4}{n-2}}_j \bar{g}_j=g_{\mathbb E}.
\ee
Furthermore for fixed $A>0$ and $D>0$, if $m_j \to 0$ then
\be \label{Omega_j-smoothly}
(\Omega'_j, g_{\mathbb E}) \to (B_0(\rho_A), g_{\mathbb E}) \textrm{ smoothly as } j \to \infty.
\ee
\end{prop}

To prove the volume preserving intrinsic flat convergence statement \eqref{prop-main-thm-VF},
we will show
they have diffeomorphic subregions $W_j$ and $W'_j$
where the metric tensors are $C^0$ close and
the volumes not covered by these subregions are small.  This will be stated more precisely below. Based on the examples of Section \ref{sec2}, we cannot expect the region near the
asymptotically cylindrical end of $(\Omega_j, \bar{g_j})$ to be $C^0$ close to $(\Omega'_j, g_{\mathbb E})$.
We will therefore choose an
\be
A_\epsilon =A_\epsilon(A, D) >0,
\ee
and then cut out the annular domain $\Omega^j_{A_\epsilon, A}$ from inside $\Omega_j$ and $\Omega'_j$
to create regions where we will have the appropriate control. Set
\be \label{W-j}
W_j=\Omega_{A_\epsilon, A}^j \,\,\subset\,\, \Omega_A^j \cap T_D(\Sigma_A^j) \,\,\subset\,\, \bar{M}_j^n,
\ee
and
\be \label{W'-j}
W'_j = \Omega_{A_\epsilon, A}^j \subset   {\mathbb{E}}^n.
\ee
The precise choice of $A_\epsilon$ will be made later in Subsection~\ref{proof-of-prop-main-thm-VF}.

\subsection{Estimates on $W_j$}

To begin the proof of Proposition~\ref{prop-main-thm-VFS} we apply the
results of the previous section to establish uniform closeness between $\bar{g}_j$ and $g_{\mathbb{E}}$ on diffeomorphic domains. This is a primary step
towards proving the intrinsic flat distance estimate \eqref{prop-main-thm-VF}.
Secondly we show the smooth convergence to a ball in $\mathbb{E}^n$ (\ref{Omega_j-smoothly}).

\begin{lem} \label{prop-subdiffeo-1}
Given any $\varepsilon>0$  and fixed
$A>A_\epsilon>0$ there exists a $\delta=\delta(\varepsilon, A_\epsilon)>0$ such that
\be
\bar{g}_j \le (1+\varepsilon)^2 {g_{\mathbb E}}\quad \textrm{ and }\quad
{g_{\mathbb E}} \le (1+\varepsilon)^2 \bar{g}_j\quad
\textrm{ on }\quad
\Omega^j_{A_\epsilon, A},
\ee
whenever the mass $m_j <\delta$.
\end{lem}

\begin{proof}
Choose $\tilde{A}_0<A_{\epsilon}$ and $\tilde{A}>A$. Then, according to Proposition \ref{C0-outside}, there is a uniform constant $C$ such that
\begin{equation}
\parallel u_j-1\parallel_{C^{0}
\left(\Omega_{\tilde{A}_{0},\tilde{A}}\right)}
\leq \frac{C m_j^{\frac{1}{4}}}{\sqrt{\tilde{A}_{0}}},
\end{equation}
where $\Omega_{\tilde{A}_0,\tilde{A}}\subset\mathbb{E}^n$ is the annular domain whose boundary consists of two coordinate spheres having areas $\tilde{A}_0<\tilde{A}$. Thus, for all $m_j$ sufficiently small we find that $\Omega_{A_{\epsilon},A}^j\subset \Omega_{\tilde{A}_0,\tilde{A}}$ and
\begin{equation}
(1+\varepsilon)^{-2}\leq u^{\frac{4}{n-2}}_j\leq (1+\varepsilon)^2,
\end{equation}
from which the desired result follows.
\end{proof}

\begin{lem} \label{lem-Omega_j-smoothly}
Fix $A>0$ and $D>0$, and let the symmetry spheres $\Sigma^j_A\subset \bar{M}_j^n$ be defined by
\be
\area_{g_j}(\Sigma^j_A)=\area_{\bar{g}_j}(\Sigma^j_A)=A.
\ee
If $m_j \to 0$ then
\be \label{area-delta-to-A}
\area_{g_{\mathbb E}}(\Sigma^j_A) \to A,
\ee
and so the spheres
\be\label{rt}
(\Sigma^j_A, {g_{\mathbb E}})\to (\partial B_0(\rho_A), {g_{\mathbb E}}) \textrm{ smoothly}
\ee
and the balls
\be \label{Omega_j-smoothly-2}
(\Omega'_j, {g_{\mathbb E}}) \to (B_0(\rho_A), {g_{\mathbb E}}) \textrm{ smoothly}.
\ee
\end{lem}

\begin{proof}
Observe that \eqref{area-delta-to-A} follows from the $C^0$ convergence
of the metric tensors in Lemma~\ref{prop-subdiffeo-1}. Moreover \eqref{rt} and \eqref{Omega_j-smoothly-2} follow immediately from \eqref{area-delta-to-A}.
\end{proof}

\subsection{Applying the Lakzian-Sormani Theorem}

In work of Lakzian and the third author \cite{Lakzian-Sormani},
the following theorem was proven providing a concrete means to estimate the intrinsic flat distance.
Intuitively this theorem observes that two manifolds $\Omega_j$ and $\Omega'_j$,
as in \eqref{Omega-j} and \eqref{Omega'-j},
are close in the intrinsic flat sense if
they have diffeomorphic subregions $W_j$ and $W'_j$, as in \eqref{W-j} and \eqref{W'-j} where the metric tensors are $C^0$ close and
the volumes not covered by these subregions are small.   See Figure~\ref{fig-Lakzian-Sormani}. When using this result for the current problem note that a dictionary between the notation is $\Omega=\Omega_j$, $\Omega'=\Omega'_j$, $W=W_j$, $W'=W'_j$, $g=\bar{g}_j$, and $g'=g_{\mathbb{E}}$.

 \begin{figure}[h] 
   \centering
   \includegraphics[width=5in]{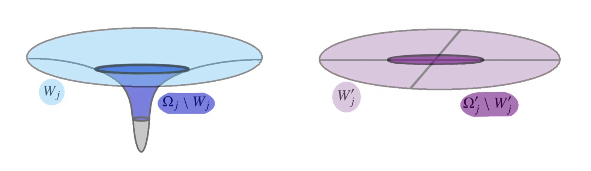}
   \caption{To prove that two regions $\Omega_j$ and $\Omega'_j$ are close in the intrinsic flat sense using the Lakzian-Sormani theorem, one first identifies subregions $W_j \subset \Omega_j$ and $W'_j\subset \Omega'_j$ that are $C^0$ close and then shows that the volumes of the excesses $\Omega_j\setminus W_j$ and $\Omega'_j\setminus W'_j$ are small. It must also be ensured that the distance distortions are small. See Theorem \ref{thm-subdiffeo}.}
   \label{fig-Lakzian-Sormani}
\end{figure}

\begin{thm} \label{thm-subdiffeo} {\em \cite{Lakzian-Sormani}}
Suppose that $(\Omega,g)$ and $(\Omega',g')$ are oriented
precompact Riemannian manifolds
with diffeomorphic subregions $W \subset \Omega$ and $W' \subset \Omega'$.
Identifying $W=W'$ assume that
\be \label{thm-subdiffeo-1}
g \le (1+\varepsilon)^2 g' \quad\textrm{ and }\quad
g' \le (1+\varepsilon)^2 g \quad\textrm { on }\quad W.
\ee
Taking the extrinsic diameters
\be \label{DU}
\max\{\diam_{g}(\Omega) , \diam_{g'} (\Omega')\}\le D_0
\ee
define a hemispherical width
\be \label{thm-subdiffeo-3}
\omega>\frac{\arccos(1+\varepsilon)^{-1} }{\pi}D_0 \quad\textrm{ which goes to } 0 \textrm{ as } \varepsilon \to 0.
\ee
Taking the difference in distances with respect to the outside manifolds,
set
\be \label{lambda}
\lambda=\sup_{x,y \in W}
|d_{\Omega, g}(x,y)-d_{\Omega', g'}(x,y)| \le 2D_0
\ee
and define the height
\be \label{thm-subdiffeo-5}
\Lambda=\,\max\left\{\sqrt{2\lambda D_0}, \,\, D_0\sqrt{\varepsilon^2 + 2\varepsilon}\right \}.
\ee
Then
\begin{align}\label{ifestimate}
\begin{split}
d_{\mathcal{F}}(\Omega, \Omega') \le &
\left(2\Lambda + \omega\right) \Big(
\vol_{g}(W)+\vol_{g'}(W')+\area_{g}(\partial W)+\area_{g'}(\partial W')\Big)\\
&+\vol_{g}(\Omega\setminus W)+\vol_{g'}(\Omega'\setminus W').
\end{split}
\end{align}
\end{thm}


\subsection{Estimating the Volumes}

In this subsection all volumes and areas appearing in the first line of \eqref{ifestimate} will be shown to be uniformly bounded, and those in the second line will be shown to be arbitrarily small. In addition, it will be established that the volumes $\vol_{\bar{g}_j}(\Omega_j)$ converges to that of a ball in Euclidean space.

\begin{lem}\label{lem-vol-Omega-Not-W}
Volumes outside the diffeomorphic subregions may be estimated by
\be
\vol_{\bar{g}_j}(\Omega_j\setminus W_{j}) \le D A_\epsilon,
\ee
and
\be
\vol_{g_{\mathbb{E}}}(\Omega'_j\setminus W'_j) \le  (1+\varepsilon)^{n-1} D A_\epsilon.
\ee
In particular, both are less than $2^{n-1}DA_{\epsilon}$.
\end{lem}

\begin{proof}
First observe that
\be
\Omega_j\setminus W_{j} = \Omega^j_A \cap T_D(\Sigma_A^j) \setminus \Omega^j_{A_\epsilon,A}
= \Omega_{A_\epsilon}^j \cap T_D(\Sigma_A^j).
\ee
By Lemma~\ref{Monotonicity}, the largest area of the radial levels in this set is
$A_\epsilon$.  Applying the coarea formula to these levels and the fact that the
depth of the set is $D$, we have
\be
\vol_{\bar{g}_j}(\Omega_j\setminus W_{j}) \le D A_\epsilon.
\ee
Applying the coarea formula again and using \eqref{thm-subdiffeo-1} yields
\be
\vol_{g_{\mathbb E}}(\Omega'_j\setminus W'_j) \le D  \vol_{{g_{\mathbb E}}}(\Sigma_{A_\varepsilon}^j)
 \le D (1+\varepsilon)^{n-1} \vol_{\bar{g}}(\Sigma_{A_\epsilon}^j)=D (1+\varepsilon)^{n-1}A_\epsilon.
\ee
\end{proof}

\begin{lem}\label{lem-vols}
Volumes of the diffeomorphic subregions may be estimated by
\be
\vol_{\bar{g}_j}(W_{j}) \le D A,
\ee
and
\be
\vol_{g_{\mathbb E}}( W'_j) \le  (1+\varepsilon)^{n-1}DA.
\ee
In particular, both are less than $2^{n-1}DA$.
\end{lem}

\begin{proof}
This follows from Lemma~\ref{Monotonicity} and the coarea formula
\be
\vol_{\bar{g}_j}(W_{j}) \le \vol_{\bar{g}_j}(\Omega_{j}) \le D \vol_{\bar{g}}(\Sigma_A^j)=D A.
\ee
Furthermore
\be
\vol_{g_{\mathbb E}}( W'_j) \le D \vol_{{g_{\mathbb E}}}(\Sigma_A^j) =D (1+\varepsilon)^{n-1}A.
\ee
\end{proof}

\begin{lem}\label{lem-areas}
Boundary areas of the diffeomorphic subregions may be estimated by
\be
\area_{\bar{g}_j}(\partial W_{j}) \le A_\epsilon + A   \le   2A,
\ee
and
\be
\area_{g_{\mathbb E}}( \partial W'_j) \le (1+\varepsilon)^{n-1} (A_\epsilon + A) \le 2 (1+\varepsilon)^{n-1} A.
\ee
In particular, both are less than $2^{n}A$.
\end{lem}

\begin{proof}
We know that $\partial W_j$ has at most two components, both of which are
radial levels of area less than $A$ by monotonicity.
The same holds for $\partial W'_j$, except that as above
the upper bound on the outer area is $(1+\varepsilon)^{n-1} A$.
\end{proof}

The next lemma will be used to obtain \textit{volume preserving} intrinsic flat convergence. It follows from the last few lemmas.

\begin{lem} \label{lem-vol-conv}
The difference of total volumes may be estimated by
\be
|\vol_{\bar{g}_j}( \Omega_j) - \vol_{g_{\mathbb E}} (\Omega'_j ) |
<  \left((1+\varepsilon)^n - 1\right) D A + 2D (1+\varepsilon)^{n-1} A_\epsilon.
\ee
\end{lem}

\begin{proof}
From Lemma~\ref{lem-vol-Omega-Not-W} we have
\be
|\vol_{g_{\mathbb E}}(\Omega'_j)-\vol_{g_{\mathbb E}}( W'_j) | \le D (1+\varepsilon)^{n-1} A_\epsilon,
\ee
and
\be
|\vol_{\bar{g}_j}(\Omega_j)- \vol_{\bar{g}_j}(W_j)|\le D A_\epsilon <D (1+\varepsilon)^{n-1} A_\epsilon.
\ee
Moreover by (\ref{thm-subdiffeo-1})
\be
|\vol_{g_{\mathbb E}}( W'_j)-\vol_{\bar{g}_j}(W_j)| \le \left((1+\varepsilon)^n-1\right) \vol_{\bar{g}_j}( W_j)
\le \left((1+\varepsilon)^n-1\right) D A.
\ee
\end{proof}

\subsection{Estimating Distances and Diameters}

\begin{lem}\label{lem-diam}
Let $D \ge \rho_A$, then
\be
\max\{\diam_{\bar{g}_j}(\Omega_j) , \diam_{{g_{\mathbb E}}}(\Omega'_j)\}\le D_0
\ee
where $D_0 \le 4\pi D$.
\end{lem}

\begin{proof}
This follows because the depth of the tubular neighborhood is $D$,
and the largest symmetry sphere satisfies
\be
\diam_{\bar{g}_j}(\Sigma_A^j) = \pi \rho_A \le \pi D.
\ee
So by the triangle inequality the diameter of $\Omega_j$ is no larger than $ 2D + \pi D \le 4\pi D$.
In addition
\be
\diam_{{g_{\mathbb E}}}(\Sigma_A^j) \le (1+\varepsilon)\diam_{\bar{g}_j}(\Sigma_A^j)=(1+\varepsilon)\pi D.
 \ee
So by the triangle inequality the diameter of $\Omega'_j$ is no larger than $ 2D + (1+\varepsilon)\pi D \le 4\pi D$.
\end{proof}

\begin{lem}\label{lem-lambda}
The difference of distances satisfies
\be
\lambda_j=\sup_{x,y \in W_j}
|d_{\bar{g}_j}(x,y)-d_{{g_{\mathbb E}}}(x,y)| \,\,\le\,\,  (1+\varepsilon) \pi \rho_{A_\epsilon} + 4\pi \varepsilon D,
\ee
where $W_j$ is identified with $W'_j$.
\end{lem}

\begin{proof}
Let $\sigma=\sigma_{x,y}$ be a line segment from $\sigma(0)=x$ to $\sigma(1)=y$ in $\Omega'_j$ so that
\be
l_{{g_{\mathbb E}}} (\sigma)\,\,=\,\,|x-y|\,\,=\,\,d_{{g_{\mathbb E}}} (x,y).
\ee
Then there is a first time that $\sigma(t_1)\in \Sigma_{A_\epsilon}^j$
and a last time that $\sigma(t_2)\in \Sigma_{A_\epsilon}^j$.  See Figure~\ref{fig-lambda}.
\begin{figure}[h] 
\centering
\includegraphics[width=.3\textwidth]{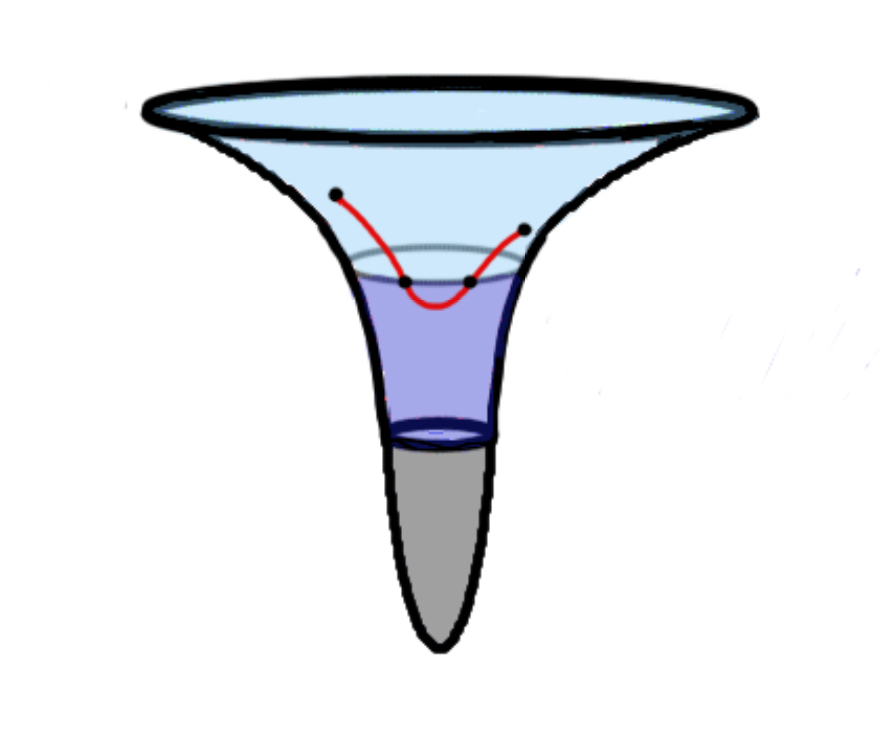}
\caption{
      }
   \label{fig-lambda}
\end{figure}
By triangle inequality we have
\begin{align}
\begin{split}
d_{\bar{g}_j}(x,y)\,\,\,\le & \,\,\,d_{\bar{g}_j}(\sigma(0), \sigma(t_1)) + d_{\bar{g}_j}(\sigma(t_1), \sigma(t_2)) +d_{\bar{g}}(\sigma(t_2), \sigma(1))  \\
\le &\,\,\, l_{\bar{g}_j}\left(\sigma([0,t_1])\right) +  \pi \rho_{A_\epsilon}   + l_{\bar{g}_j}\left(\sigma([t_2,1])\right)\\
\le &\,\,\, (1+\varepsilon)\, l_{{g_{\mathbb E}}} \left(\sigma([0,t_1])\right) +\pi \rho_{A_\epsilon}  +(1+\varepsilon)\, l_{{g_{\mathbb E}}}\left(\sigma([t_2,1])\right) \\
\le &\,\,\, \pi \rho_{A_\epsilon}   + (1+\varepsilon) l_{{g_{\mathbb E}}} \left(\sigma([0,1])\right)\\
= & \,\,\,\pi \rho_{A_\epsilon}   + (1+\varepsilon) d_{{g_{\mathbb E}}} (x,y)\\
\le & \,\,\,\pi \rho_{A_\epsilon}   + \varepsilon D_0 +  d_{{g_{\mathbb E}}} (x,y).
\end{split}
\end{align}
Here we have used Lemma \ref{lem-diam} and the fact that $\Sigma_{A_\epsilon}^j$ is connected.

Let $\gamma=\gamma_{x,y}$ be a $\bar{g}_j$-length minimizing geodesic from $\gamma(0)=x$ to $\gamma(1)=y$ in $\Omega_j$ so that
\be
l_{\bar{g}_j} (\gamma)=d_{\bar{g}_j} (x,y).
\ee
Then there is a first time that $\gamma(t_1)\in \Sigma_{A_\epsilon}^j$
and a last time that $\gamma(t_2)\in \Sigma_{A_\epsilon}^j$.
By the triangle inequality we have
\begin{align}
\begin{split}
\qquad d_{g_{\mathbb E}}(x,y)\,\,\,\le & \,\,\,d_{{g_{\mathbb E}}}(\gamma(0), \gamma(t_1)) + d_{{g_{\mathbb E}}}(\gamma(t_1), \gamma(t_2)) +d_{{g_{\mathbb E}}}(\gamma(t_2), \gamma(1))  \\
\le &\,\,\, l_{{g_{\mathbb E}}}\left(\gamma([0,t_1])\right) +  \diam_{g_{\mathbb E}}(\Sigma_{A_\epsilon}^j)   + l_{{g_{\mathbb E}}}\left(\gamma([t_2,1])\right)\\
\le &\,\,\,(1+\varepsilon) l_{\bar{g}_j}\left(\gamma([0,t_1])\right) + (1+\varepsilon)\diam_{\bar{g}_j}(\Sigma_{A_\epsilon}^j)
+ (1+\varepsilon) l_{\bar{g}_j}\left(\gamma([t_2,1])\right)\\
\le &\,\,\, (1+\varepsilon)\pi \rho_{A_\epsilon}   + (1+\varepsilon) l_{\bar{g}_j} \left(\gamma([0,1])\right)\\
= & \,\,\,(1+\varepsilon)\pi \rho_{A_\epsilon}   + (1+\varepsilon) d_{\bar{g}_j} (x,y)\\
\le &\,\,\, (1+\varepsilon)\pi \rho_{A_\epsilon}   + \varepsilon D_0 +  d_{\bar{g}_j} (x,y).
\end{split}
\end{align}
Combining these with $D_0 \le 4 \pi D$ from Lemma~\ref{lem-diam} we obtain the desired result.
\end{proof}

\subsection{Proof of Proposition~\ref{prop-main-thm-VFS}}
\label{proof-of-prop-main-thm-VF}

First note that by Lemmas~\ref{lem-vol-Omega-Not-W}-\ref{lem-diam} applied to Theorem~\ref{thm-subdiffeo} we have
\be
d_{\mathcal{F}}(\Omega_j, \Omega'_j) \le
\left(2\Lambda_j + \omega\right) \Big(2 D 2^{n-1}A + 2 A 2^n\Big) + 2 D 2^{n-1} A_\epsilon,
\ee
where
\begin{equation}
\Lambda_j=\,\max\left\{\sqrt{2\lambda_j D_0}, \,\, D_0\sqrt{\varepsilon^2 + 2\varepsilon}\right \},
\end{equation}
Furthermore by Lemma~\ref{lem-vol-conv}
\be
|\vol_{\bar{g}_j}( \Omega_j) - \vol_{g_{\mathbb E}} ( \Omega_j ) |
<  \left((1+\varepsilon)^n - 1\right) D A + D 2^{n} A_\epsilon.
\ee

Given $A>0$, $D>0$, and $\epsilon>0$ we must choose $\delta=\delta(A, D,\epsilon)>0$ sufficiently small so that
\be \label{what-we-want}
d_{\mathcal{VF}}(\Omega_j, \Omega'_j)\,\,=\,\,
d_{\mathcal{F}}(\Omega_j, \Omega'_j)+|\vol_{\bar{g}_j}( \Omega_j) - \vol_{g_{\mathbb E}} ( \Omega_j ) | <\epsilon
\ee
whenever $\delta<m_j$.  Therefore choose $\varepsilon$ and $A_\epsilon$ appropriately in order to satisfy
\be \label{var-1}
D 2^{n} A_\epsilon< \epsilon/4,
\ee
\be \label{var-2}
 \left((1+\varepsilon)^n - 1\right) D A <\epsilon/4,
\ee
and
\be \label{var-4}
\left(2\Lambda_j + \omega\right) C_{A,D} < \epsilon/4,
\ee
where
\be
C_{A,D}=(2 D 2^{n-1}A + 2 A 2^n ).
\ee
Observe that these together imply \eqref{what-we-want}. It remains only to show that \eqref{var-4} may be obtained from choosing $\varepsilon$ and $A_{\epsilon}$ small.

According to the definition of $\Lambda_j$ and $\omega$, \eqref{var-4} will be valid if
\be \label{var-4epsilon}
2D_0\sqrt{\varepsilon^2 + 2\varepsilon} C_{A,D}< \epsilon/8,
\ee
\be \label{var-4lambda}
2\sqrt{2\lambda_j D_0} C_{A,D}< \epsilon/8,
\ee
and
\be \label{var-4a}
\omega C_{A,D}=2\frac{\arccos(1+\varepsilon)^{-1} }{\pi}D_0 C_{A,D}< \epsilon/8.
\ee
Moreover, from Lemma~\ref{lem-lambda} it follows that \eqref{var-4lambda} holds if
\be\label{var-4lambda12}
2\sqrt{8 \pi \varepsilon D D_0} C_{A,D}< \epsilon/16,
\ee
and
\be\label{var-4lambda1}
2\sqrt{ 4 \pi \rho_{A_\epsilon}  D_0} C_{A,D}< \epsilon/16.
\ee
Thus we choose the area of the inner level set
$A_\epsilon=A_\epsilon(A,D)>0$ small enough so that \eqref{var-1}, and \eqref{var-4lambda1} hold.  Recall that $\rho_{A_\epsilon}$ is the radius of the $n-1$-sphere whose area is $A_\epsilon$.
In addition, we choose $\varepsilon=\varepsilon(A, D,\epsilon)$
small enough so that \eqref{var-2}, \eqref{var-4epsilon},
\eqref{var-4a}, and \eqref{var-4lambda12} hold.
It follows that there exists
$\delta=\delta(\varepsilon, A_\epsilon)>0$
such that $m_j < \delta$ implies
$d_{\mathcal{VF}}(\Omega_j, \Omega'_j)< \epsilon$.
Finally \eqref{Omega_j-smoothly} was established in Lemma~\ref{lem-Omega_j-smoothly}.

\section{Convergence of the Second Fundamental Form Differences}
\label{sec7}

In Theorem \ref{main-thm} a statement \eqref{main-thm-kL2} is given concerning the convergence of the difference of the second fundamental forms of the initial data and Jang graphs. In this section we give a proof of that statement.

\subsection{Outside Control}

The pointwise control on the conformal factor gives rise to $L^2$ control of the difference of second fundamental forms. Recall that $\Omega_{A}$ is the domain inside the symmetry sphere of volume $A$ in $(\bar{M}^n,\bar{g})$ and $(M^n,g)$.
By combining Proposition \ref{C0-outside} and \eqref{stability1} we obtain the following.

\begin{prop}\label{2ndFundFormOutside}
Let $(M^n_j,g_j,k_j)$ be a sequence of spherically symmetric, uniformly asymptotically flat initial data satisfying the dominant energy condition and
with either outermost apparent horizon boundary or no boundary.
For every $A_{0}>0$ there exists a uniform constant $C$ such that
\begin{equation}
\parallel h_j-k_j\parallel^2_{L^2(\bar{M}_j^n\setminus\Omega_{A_{0}}^j,\bar{g}_j)}
\,\,\leq \,\,C m_j.
\end{equation}
\end{prop}

\begin{proof}
According to Proposition \ref{C0-outside} and the asymptotics of $u_j$, we find by choosing $0<\tilde{A}_0<A_0$ that
\begin{equation}\label{llllp}
|u_j-1|\leq c_1 m_{j}^{\frac{1}{4}}\quad\quad\text{ on }\quad\quad \bar{M}^n_j\setminus\Omega^j_{A_0},
\end{equation}
for some uniform constant $c_1$. Thus for large $j$, the inequality \eqref{stability1} implies
\begin{equation}
c_{n}^{-1}(n-2)\omega_{n-1} m_j
\,\,\geq\,\,
\int_{\bar{M}_j^n}|h_j-k_j|_{\bar{g}_j}^2 u_j^2
dV_{\bar{g}_j}
\,\,\geq\,\,\tfrac{1}{2}\int_{\bar{M}_j^n\setminus\Omega_{A_{0}}^j}
|h_j-k_j|_{\bar{g}_j}^2 dV_{\bar{g}_j}.
\end{equation}
\end{proof}

\subsection{Inside Control}

The conformal factors $u_j$ decay exponentially fast along the cylindrical ends of $(\bar{M}^n_j,\bar{g}_j)$, and thus the techniques used above to obtain estimates outside a fixed level surface do not extend down the cylindrical end. Nevertheless, with an extra hypothesis of uniform $L^2$ control for the initial data second fundamental forms $k_j$, we are able to obtain uniform $L^2$ bounds for the difference of second fundamental forms on tubular neighborhoods of the anchor surface. It may be possible to weaken this hypothesis, however we expect that some condition on $k_j$ is necessary in order to control the difference of second fundamental forms.

\begin{prop}\label{L2 bound1}
Let $(M^n_j,g_j,k_j)$ be a sequence of spherically symmetric, uniformly asymptotically flat initial data satisfying the dominant energy condition and
with either outermost apparent horizon boundary or no boundary. Further, assume
the uniform bound $\parallel k_j\parallel _{L^2(M_j^n)}\leq B$. Given $A>0$ and $D>0$ there exists $C$ depending only on $A$, $B$, and $D$ such that
\begin{equation}
	\parallel h_j-k_j\parallel_{L^2(\Omega_A^j \cap T_{D}(\Sigma_A^j),\bar{g}_j)}
\,\,\leq\,\, C.
\end{equation}
\end{prop}

\begin{proof}
For convenience in this proof we will drop the subscript $j$.
Recall from \eqref{stability} that
\begin{equation}
\int_{\bar{M}^n}|h-k|_{\bar{g}}^2\phi^2
dV_{\bar{g}}\,\,\leq\,\,
\int_{\bar{M}^n}\left(c_{n}^{-1}|\nabla\phi|_{\bar{g}}^2
+R_{\bar{g}}\phi^2\right)dV_{\bar{g}}
\end{equation}
for all $\phi\in C^{\infty}_{c}(\bar{M}^n)$. Choose a particular test function such that $0\leq \phi\leq 1$, $\phi\equiv 1$ on $T_{D+1}(\Sigma_A)$, and $|\nabla\phi|_{\bar{g}}\leq 2$ then
\begin{equation}
\int_{T_D(\Sigma_A)}|h-k|_{\bar{g}}^2
dV_{\bar{g}}
\,\,\leq\,\, 2c_n^{-1}\text{Vol}_{\bar{g}}(T_{D+1}(\Sigma_A))+
\int_{T_{D+1}(\Sigma_A)}R_{\bar{g}}\phi^2 dV_{\bar{g}}.
\end{equation}

Write
\begin{equation}
	\bar{g}=d\bar{s}^2+\rho(\bar{s})^2 g_{S^{n-1}}
\end{equation}
such that the interval $0<\bar{s}<D+2$ covers $T_{D+1}(\Sigma_A)$, and notice
that
\begin{equation}	 R_{\bar{g}}=-2(n-1)\frac{\ddot{\rho}}{\rho}+(n-1)(n-2)\frac{1-\dot{\rho}^2}{\rho^2}.
\end{equation}
Integrating by parts produces
\begin{align}\label{Sclr crvtr int}
\begin{split}
&\int_{T_{D+1}(\Sigma_A)} R_{\bar{g}}\phi^2 dV_{\bar{g}} \\
=&\,\,\omega_{n-1}\int_{0}^{D+2}
\left[-2(n-1)\frac{\ddot{\rho}}{\rho}+(n-1)(n-2)
\frac{1-\dot{\rho}^2}{\rho^2}\right]\phi^2\rho^{n-1}d\bar{s}\\
=&\,\,\omega_{n-1}\int_{0}^{D+2}\left[4(n-1)\phi\dot\phi\rho^{n-2}\dot\rho
+(n-1)(n-2)\phi^2\rho^{n-3}\dot\rho^2+(n-1)(n-2)\phi^2\rho^{n-3}\right]d\bar{s}.
\end{split}
\end{align}
Since the mean curvature of coordinate spheres is $\bar{H}=(n-1)\rho^{-1}\dot\rho$, this may be rewritten as
\begin{align}\label{qwew}
\begin{split}
\int_{T_{D+1}(\Sigma_A)}R_{\bar{g}}\phi^2 dV_{\bar{g}}
=&\int_{T_{D+1}(\Sigma_A)}\left(4\langle \phi\nabla\phi,\partial_{\bar{s}}\rangle
\bar{H}  +\frac{n-2}{n-1}\phi^2 \bar{H}^2 \right)dV_{\bar{g}}\\
&+\int_{0}^{D+2}\omega_{n-1}(n-1)(n-2)\phi^2 \rho^{n-3}d\bar{s}.
\end{split}
\end{align}

We now estimate each term of \eqref{qwew} separately. By uniform asymptotic flatness of $\bar{g}$ in Lemma \ref{Unf-Asym-Flat for bar} and asymptotic control \eqref{llllp} of $u$, we have that
\begin{equation}
\rho(D+2)=\left(\frac{A(D+2)}{\omega_{n-1}}\right)^{\frac{1}{n-1}}
\end{equation}
is uniformly bounded and thus
\begin{equation}
\int_{0}^{D+2}\phi^2 \rho^{n-3}d\bar{s}
\,\leq\,(D+2)\left(\frac{A(D+2)}{\omega_{n-1}}\right)^{\frac{n-3}{n-1}}
\,\leq\, c_1.
\end{equation}
Next according to Lemma \ref{lem-vol-conv} it holds that $\text{Vol}_{\bar{g}}(T_{D+1}(\Sigma))$ is uniformly bounded, and hence
by H\"{o}lder's inequality
\begin{equation}
\int_{T_{D+1}(\Sigma_A)}\langle \phi\nabla\phi,\partial_{\bar{s}}\rangle
\bar{H} dV_{\bar{g}}
\,\leq\, 2\text{Vol}_{\bar{g}}(T_{D+1}(\Sigma))^{\frac{1}{2}}
\left(\int_{T_{D+1}(\Sigma_A)}\bar{H}^2 dV_{\bar{g}}\right)^{\frac{1}{2}}
\leq\, c_2 \parallel \bar{H}\parallel_{L^2(T_{D+1}(\Sigma_A),\bar{g})}.
\end{equation}
The remaining term of \eqref{qwew} may also be estimated by the $L^2$ norm of $\bar{H}$.

To complete the proof we will use the extra hypothesis concerning $k$ to bound $\bar{H}$ in $L^2$. First observe that
\begin{equation}
\bar{g}=g+df^2=(1+f_s^2)ds^2 +\rho(s)^2 g_{S^{n-1}}\quad \Rightarrow
\quad \frac{d\bar{s}}{ds}=\sqrt{1+f_s^2},
\end{equation}
so that
\begin{equation}
\bar{H}=(n-1)\rho^{-1}\frac{d\rho}{d\bar{s}}
=(n-1)\rho^{-1}\frac{d\rho}{ds}\frac{ds}{d\bar{s}}=\frac{1}{\sqrt{1+f_s^2}}H
\end{equation}
where $H$ is the mean curvature with respect to $g$. Furthermore, let $\psi:T_{D+1}(\Sigma_A)\rightarrow M^n$ be the diffeomorphism onto its image associated with the coordinate change $\bar{s}\rightarrow s$. Then
\begin{align}\label{yyyy}
\begin{split}
\int_{T_{D+1}(\Sigma_A)}\bar{H}^2 dV_{\bar{g}}=&
\int_{T_{D+1}(\Sigma_A)}\left[\bar{H}^2
-\frac{(\text{Tr}_{S_{\bar{s}}}k)^2}{1+f_s^2}\right]dV_{\bar{g}}
+\int_{T_{D+1}(\Sigma_A)}\frac{(\text{Tr}_{S_{\bar{s}}}k)^2}{1+f_s^2} dV_{\bar{g}}\\
=&\int_{0}^{D+2}(1+f_s^{2})^{-1}\int_{S_{\bar{s}}}
\left[H^2-(\text{Tr}_{S_{\bar{s}}}k)^2\right]dA_{\bar{g}} d\bar{s}
+\int_{\psi[T_{D+1}(\Sigma_A)]}
\frac{(\text{Tr}_{S_{\bar{s}}}k)^2}{\sqrt{1+f_s^2}}dV_g\\
\leq &\int_{0}^{D+2}(n-2)^2\omega_{n-1}^{\frac{n-3}{n-2}}
A(\bar{s})^{\frac{n-3}{n-2}}d\bar{s}
+\int_{M^n}(n-1)^2 |k|_{g}^2 dV_g,
\end{split}
\end{align}
where we have used that the spacetime Hawking mass
\begin{equation}
m(s):=\frac{1}{2}\left(\frac{A(s)}{\omega_{n-1}}\right)^{\frac{n-2}{n-1}}
\left[1-\frac{1}{(n-1)^2 \omega_{n-1}^{\frac{2}{n-1}}A(s)^{\frac{n-3}{n-1}}}
\int_{S_{s}}\left(H^2-(\text{Tr}_{S_s}k)^2\right)dA_g\right]
\end{equation}
is nonnegative \cite{Khuri}, and the areas of coordinate spheres agree with respect to both metrics $g$ and $\bar{g}$, that is $A(s)=A(\bar{s})$. Again using that $\rho(D+2)$ is uniformly bounded produces
\begin{equation}
\int_{T_{D+1}(\Sigma_A)}\bar{H}^2 dV_{\bar{g}}\leq (D+2)(n-2)^2\omega_{n-1}^{\frac{n-3}{n-2}}c_3
+(n-1)^2\parallel k\parallel_{L^2(M^n)}^2,
\end{equation}
from which the desired result follows.
\end{proof}

\subsection{Global Control}

Here we show that the difference of second fundamental forms tends to zero in $L^p$, $1\leq p<2$ on the tubular neighborhoods
$\Omega_A \cap T_{D}(\Sigma_A)\subset \bar{M}^n$, when the masses go to zero. This result, and in fact even a uniform bound, cannot be extended to the entire domain $\Omega_A$. The reason is due to the cylindrical asymptotics contained within this domain, and the observation that in the limit along the cylinder $h-k$ converges to $II-i^* k$, where $II$ is the second fundamental form of $\partial M^n$ in $M^n$ and $i:\partial M^n\hookrightarrow M^n$ is inclusion. Since $II-i^* k$ does not necessarily vanish, it follows that $h-k$ is not necessarily in $L^p(\Omega_A)$.

\begin{thm}\label{thmhk}
Let $(M^n_j,g_j,k_j)$ be a sequence of spherically symmetric, uniformly asymptotically flat initial data satisfying the dominant energy condition and
with either outermost apparent horizon boundary or no boundary. Further, assume
the uniform bound $\parallel k_j\parallel _{L^2(M_j^n)}\leq B$ and that $m_j\rightarrow 0$. Then for any $1\leq p<2$, $A>0$, and $D>0$
\begin{equation}
	\parallel h_j-k_j\parallel_{L^p(\Omega_A^j \cap T_{D}(\Sigma_A^j),\bar{g}_j)}
\rightarrow 0.
\end{equation}
\end{thm}

\begin{proof}
Let $A_{\varepsilon}>0$ be fixed. Then by H\"{o}lder's inequality
\begin{align}
\begin{split}
&\int_{\Omega^j_{A}\cap T_{D}(\Sigma^j_A)} |h_j-k_j|_{\bar{g}_j}^p dV_{\bar{g}_j}\\
=&\,\,\int_{\left[\Omega^j_A\cap T_{D}(\Sigma^j_A)\right]\cap\Omega^j_{A_{\varepsilon}}}
|h_j-k_j|_{\bar{g}_j}^p dV_{\bar{g}_j}
+\int_{\left[\Omega^j_A\cap T_{D}(\Sigma^j_A)\right]\setminus\Omega^j_{A_{\varepsilon}}}|h_j-k_j|_{\bar{g}_j}^p dV_{\bar{g}_j}\\
\leq& \,\,\text{Vol}_{\bar{g}_j}\left(\left[\Omega^j_A\cap T_{D}(\Sigma^j_A)\right]\cap\Omega^j_{A_{\varepsilon}}\right)^{\frac{2-p}{2}}
\left(\int_{\Omega^j_A\cap T_{D}(\Sigma^j_A)}|h_j-k_j|^2_{\bar{g}_j} dV_{\bar{g}_j}\right)^{\frac{p}{2}}\\
&\,\,+\,\text{Vol}_{\bar{g}_j}\left(\left[\Omega^j_A\cap T_{D}(\Sigma^j_A)\right]\setminus\Omega^j_{A_{\varepsilon}}\right)^{\frac{2-p}{2}}
\left(\int_{\left[\Omega^j_A\cap T_{D}(\Sigma^j_A)\right]\setminus\Omega_{A_{\varepsilon}}^j}|h_j-k_j|^2_{\bar{g}_j} dV_{\bar{g}_j}\right)^{\frac{p}{2}}.
\end{split}
\end{align}
Using the notation and results of Lemma \ref{lem-vol-Omega-Not-W}, observe that
\begin{equation}
\text{Vol}_{\bar{g}_j}\left(\left[\Omega^j_A\cap T_{D}(\Sigma^j_A)\right]\cap\Omega^j_{A_{\varepsilon}}\right)=
\text{Vol}_{\bar{g}_j}(\Omega_j\setminus W_{j}) \leq D A_\varepsilon.
\end{equation}
Furthermore by Lemma \ref{lem-vol-conv}
\begin{equation}
\text{Vol}_{\bar{g}_j}\left(\left[\Omega^j_A\cap T_{D}(\Sigma^j_A)\right]\setminus\Omega^j_{A_{\varepsilon}}\right)
\leq\text{Vol}_{\bar{g}_j}(\Omega_j)\leq C_1,
\end{equation}
by Proposition \ref{2ndFundFormOutside}
\begin{equation}
\parallel h_j-k_j\parallel_{L^2(\bar{M}_j^n\setminus\Omega_{A_{\varepsilon}}^j,\bar{g}_j)}
\,\,\leq\, \sqrt{C_3(A_{\varepsilon}) m_j},
\end{equation}
and by Proposition \ref{L2 bound1}
\begin{equation}
	\parallel h_j-k_j\parallel_{L^2(\Omega_A^j \cap T_{D}(\Sigma_A^j),\bar{g}_j)}
\,\,\leq\, C_2.
\end{equation}

It follows that
\begin{equation}
\parallel h_j-k_j\parallel_{L^p(\Omega_A^j \cap T_{D}(\Sigma_A^j),\bar{g}_j)}^p
\,\,\,\leq\,\,\, (DA_{\varepsilon})^{\frac{2-p}{2}}C_2^p
+C_1^{\frac{2-p}{2}}(C_3(A_{\varepsilon}) m_j)^{\frac{p}{2}}.
\end{equation}
This can be made arbitrarily small by choosing $A_{\varepsilon}$ small and $j$ appropriately large.
\end{proof}

\section{Proof of Theorem~\ref{main-thm}}
\label{sec8}

The proof of the main theorem follows quickly from Proposition \ref{prop-main-thm-VFS} and Theorem \ref{thmhk}. Recall that
\be
\Omega_j \,\,=\,\,\Omega^j_A \,\cap\, T_D(\Sigma^j_A) \,\,\subset\,\, \bar{M}_j^n \textrm{ with metric tensor } \bar{g}_j,
\ee
and
\be
\Omega'_j\,\,= \,\, \Omega^j_A \,\,\subset\,\, {\mathbb{E}}^n \textrm{ with metric tensor } \tilde{g}_j=u^{\frac{4}{n-2}}_j \bar{g}_j=g_{\mathbb E}.
\ee
By the triangle inequality
\begin{equation}
d_{\mathcal{VF}}\left((\Omega_j,\bar{g}_j),(B_0(\rho_A),g_{\mathbb{E}})\right)
\,\,\leq\,\, d_{\mathcal{VF}}\left((\Omega_j,\bar{g}_j),(\Omega'_j,g_{\mathbb{E}})\right)
\,+\,d_{\mathcal{VF}}\left((\Omega'_j,g_{\mathbb{E}}),
(B_0(\rho_A),g_{\mathbb{E}})\right).
\end{equation}
As $m_j\rightarrow 0$ we have that \eqref{prop-main-thm-VF} produces
\begin{equation}
d_{\mathcal{VF}}\left((\Omega_j,\bar{g}_j),(\Omega'_j,g_{\mathbb{E}})\right)
\rightarrow 0,
\end{equation}
and \eqref{Omega_j-smoothly} yields
\begin{equation}
d_{\mathcal{VF}}\left((\Omega'_j,g_{\mathbb{E}}),
(B_0(\rho_A),g_{\mathbb{E}})\right)
\rightarrow 0,
\end{equation}
since smooth convergence implies convergence in the volume preserving intrinsic flat sense. Finally, Theorem \ref{thmhk} gives \eqref{main-thm-kL2}.


\end{document}